\documentclass[11pt]{article}

\usepackage{amssymb}
\usepackage{amsmath}
\usepackage{amsthm}
\usepackage{tikz-cd}
\usepackage{pgf,acadpgf}
\usepackage{booktabs}

\long\def\commentout#1{}

\numberwithin{equation}{section}

\newtheoremstyle{slplain}
  {\topsep}
  {\topsep}
  {\slshape}
  {0pt}
  {\bfseries}
  {.}
  {0.5em}
  {}

\theoremstyle{slplain}
  \newtheorem{THM}{Theorem}[section]
  \newtheorem{LEM}[THM]{Lemma}
  \newtheorem{PROP}[THM]{Proposition}

\theoremstyle{definition}
  
  \newtheorem{EX}[THM]{Example}

\newcommand\nlongrightarrow{\longrightarrow\kern -1.45em/\kern 0.9em}

\renewcommand{\le}{\leqslant}
\renewcommand{\ge}{\geqslant}

\newcommand{\tp}{\mathrm{tp}}

\newcommand{\0}{\varnothing}

\renewcommand{\phi}{\varphi}
\renewcommand{\epsilon}{\varepsilon}
\newcommand{\BB}{\mathbf{B}}
\newcommand{\CC}{\mathbf{C}}
\newcommand{\DD}{\mathbf{D}}
\newcommand{\EE}{\mathbf{E}}

\newcommand{\KK}{\mathbf{K}}

\newcommand{\NN}{\mathbb{N}}

\newcommand{\RR}{\mathbb{R}}
\newcommand{\TT}{\mathbf{T}}

\newcommand{\union}{\cup}
\newcommand{\restr}[2]{\hbox{$#1$}\hbox{$\upharpoonright$}_{#2}}

\newcommand{\Boxed}[1]{\mbox{$#1$}}
\newcommand{\id}{\mathrm{id}}
\newcommand{\ID}{\mathrm{ID}}
\newcommand{\Ob}{\mathrm{Ob}}

\newcommand{\op}{\mathrm{op}}

\newcommand{\fin}{\mathit{fin}}
\newcommand{\full}{\mathit{full}}

\newcommand{\calA}{\mathcal{A}}
\newcommand{\calB}{\mathcal{B}}
\newcommand{\calC}{\mathcal{C}}
\newcommand{\calD}{\mathcal{D}}
\newcommand{\calE}{\mathcal{E}}

\newcommand{\calG}{\mathcal{G}}

\newcommand{\calM}{\mathcal{M}}

\newcommand{\calS}{\mathcal{S}}
\newcommand{\calT}{\mathcal{T}}
\newcommand{\calU}{\mathcal{U}}

\newcommand{\calX}{\mathcal{X}}
\newcommand{\calY}{\mathcal{Y}}

\newcommand{\im}{\mathrm{im}}
\newcommand{\Spec}{\mathrm{Spec}}

\newcommand{\Set}{\mathbf{Set}}
\newcommand{\SetSurj}{\mathbf{Set}_{\mathit{surj}}}
\newcommand{\Top}{\mathbf{Top}}

\newcommand{\CHrs}{\mathbf{FinCh}_{\mathit{rs}}}
\newcommand{\WchRs}{\mathbf{Wch}_{\mathit{rs}}}
\newcommand{\LMetNers}{\mathbf{LMet}_{\mathit{ners}}}
\newcommand{\LMetNersFin}{\mathbf{LMet}_{\mathit{ners}}^\fin}
\newcommand{\MetNes}{\mathbf{Met}_{\mathit{nes}}}

\newcommand{\Map}{\mathrm{Map}}

\newcommand{\RSurj}{\mathrm{RSurj}}
\newcommand{\Surj}{\mathrm{Surj}}
\newcommand{\QMap}{\mathrm{QMap}}
\newcommand{\Borel}[1]{\overset{\flat}{#1}}

\newcommand{\FinE}{\mathbf{FinE}}
\newcommand{\FinSetSurj}{\mathbf{FinSet}_{\mathit{surj}}}
\newcommand{\FinRelQuo}{\mathbf{FinRel}_{\mathit{quo}}}
\newcommand{\RelQuo}{\mathbf{Rel}_{\mathit{quo}}}

\title{Classes of finite relational structures over finite languages have dual Ramsey degrees}
 \author{%
    Aleksa D\v zuklevski, Dragan Ma\v sulovi\'c\\
    University of Novi Sad, Faculty of Sciences\\
    Department of Mathematics and Informatics\\
    Trg Dositeja Obradovi\'ca 3, 21000 Novi Sad, Serbia\\
    e-mail: \texttt{\{aleksa.dzuklevski,dragan.masulovic\}@dmi.uns.ac.rs}
  }

\begin{document}
\maketitle

\begin{abstract}
  Classical Ramsey theory has successfully extended to relational structures,
  yielding a wealth of results that have profoundly influenced other areas of mathematics.
  Interestingly, the same development has not occurred in the case of dual Ramsey theory.
  The main goal of this paper is to advance the dual Ramsey theory for finite relational
  structures with respect to natural structure-preserving maps. Tools from category theory
  prove instrumental in this endeavor, as was previously the case for finite algebraic systems
  where the dual Ramsey property had been established for every class of finite algebras coming from
  an equationally defined class. One cannot help but feel that dual Ramsey phenomena are deeply
  connected to categorical strategies.

  \medskip


  \medskip

\end{abstract}

\section{Introduction}
One of the first successful applications of category theory in combinatorics was
Graham, Leeb and Rothschild's 1972 proof of the Rota Conjecture \cite{graham-leeb-rothschild}
(which can be seen as a vector-space analogue of the Finite Ramsey Theorem from the 1930s \cite{Ramsey}).
Around the same time, the first results in the area of structural Ramsey theory were obtained (see \cite{N1995} for references).
As it turns out, structural Ramsey theory is yet another segment of combinatorics where the tools of category theory prove to be fruitful,
in particular in the context of dual Ramsey phenomena. Unlike ``direct'' Ramsey theory which had successfully evolved to
relational structures and supplied us with a plethora of results which have profoundly influenced other ares of mathematics
(see for example~\cite{KPT}), the same has not happened with the dual Ramsey theory, Table~\ref{drdrs.fig.rpfos}.
A notable exception is Solecki's paper~\cite{solecki-structs} where a finite dual Ramsey theorem is proved for
finite structures over first-order signatures where functional symbols are interpreted as usual finitary operations,
and relational symbols are interpreted as dual relations. Using the strategies that steam from category theory, Solecki's results for
dual Ramsey properties of finite algebras were reproduced and slightly generalized in~\cite{masul-drp-algs}, and extended to
infinite dual Ramsey statements.

\begin{table}
  \centering
  \begin{tabular}{cccc}
    \toprule
                       &  \bf Relational                              & \bf Algebraic      &   \\
                       &  \bf structures                              & \bf systems        & \bf Tools \\
    \midrule
      \bf ``Direct''   & almost any                                               &  \textsl{a few very}          & \\
      \bf Ramsey       & reasonable class                                         &  \textsl{specific instances}  & combinatorial \\
      \bf properties   & (see e.g. \cite{Hubicka-Nesetril-AllThoseRamseyClasses}) &  (see e.g.~\cite{jezek-nesetril-1983,sokic-semilattices,sokic-unary-functions,masul-coalg-unary-algs})           &         \\
    \midrule                                     
      \bf Dual         & almost any                                      &  every reaso-      &      \\
      \bf Ramsey       & reasonable class                                &  nable class                 & categorical      \\
      \bf properties   & (this paper)                                    &  (see \cite{solecki-structs,masul-drp-algs}) & \\
    \bottomrule
  \end{tabular}
  \caption{Ramsey properties in first-order structures}
  \label{drdrs.fig.rpfos}
\end{table}

The main goal of this paper is to develop the dual Ramsey theory for finite relational structures with respect to natural
structure-preserving maps. Tools of category theory prove to be instrumental in this endeavor,
as they allow us to lift the dual Ramsey phenomena from the context of chains and rigid surjections onto
the context of structures and structure-preserving maps.

To achieve this several purely categorical techniques are employed. They are introduced and adjusted to our needs in Section~\ref{drdrs.sec.prelim}.
The main results of the paper are presented in Section~\ref{drdrs.sec.rsfl} where
we show the existence of dual Ramsey degrees, both small and big, for general relational structures over finite relational languages.
In Section~\ref{drdrs.sec.bin-rel-structs} we consider some well-known classes of reflexive binary relational structures
such as reflexive graphs (that is, graphs with all loops) and posets. We show that the corresponding classes of finite structures
have dual small Ramsey degrees, and that projectively universal countably infinite structures have dual big Ramsey degrees.
All the computations are carried out with respect to quotient maps.
In sharp contrast to this, we also show that finite reflexive tournaments do not have dual small Ramsey degrees.
The easy observation that projectively universal countable tournaments do not exist renders the question
of dual big Ramsey degrees for reflexive tournaments meaningless.
We conclude the paper with Section~\ref{drdrs.sec.met-spc} where we compute dual small Ramsey degrees of finite metric spaces,
and compute dual big Ramsey degrees of finite metric spaces in arbitrary projectively universal countable
metric spaces. The natural morphisms here are non-expansive surjective maps.

Looking at Table~\ref{drdrs.fig.rpfos} again, one cannot escape the feeling that dual Ramsey phenomena
are intimately tied to strategies of category theory.

\section{Preliminaries}
\label{drdrs.sec.prelim}

\paragraph{Structures.}
A \emph{chain} is a linearly ordered set $(A, \Boxed<)$. Finite or countably infinite
chains will sometimes be denoted as $\{a_1 < \ldots < a_n < \ldots\}$. For example, $\omega = \{0 < 1 < 2 < \ldots\}$.
A chain $(A, \Boxed<)$ is \emph{well-ordered} if every nonempty subset of $A$ has the least element with respect to~$<$.
Let $(A, \Boxed<)$ and $(B, \Boxed<)$ be well-ordered chains. A surjective map $f : A \to B$ is a \emph{rigid surjection}
if $b_1 < b_2$ implies $\min f^{-1}(b_1) < \min f^{-1}(b_2)$ for all $b_1, b_2 \in B$.
We write $\RSurj\big((A, \Boxed<), (B, \Boxed<)\big)$
for the set of all rigid surjections $(A, \Boxed<) \to (B, \Boxed<)$.

A \emph{relational language} is a set $\Theta$ of \emph{relation symbols}, each of which comes with its \emph{arity}.
A \emph{$\Theta$-structure} $\calA = (A, \Theta^\calA)$ is a set $A$ together with a set $\Theta^\calA$ of
relations on $A$ which are interpretations of the corresponding symbols in $\Theta$.
The underlying set of a structure $\calA$, $\calA_1$, $\calA^*$, \ldots\ will always be denoted by its roman
letter $A$, $A_1$, $A^*$, \ldots\ respectively.
A structure $\calA = (A, \Theta^\calA)$ is \emph{finite} if $A$ is a finite set.
An element $a \in A$ is an \emph{isolated point} in $\calA = (A, \Theta^\calA)$ if for every $R \in \Theta$ and
every $(b_1, \ldots, b_h) \in R^A$ we have that $a \notin \{b_1, \ldots, b_h\}$.

Let $\calA$ and $\calB$ be $\Theta$-structures and let $f : A \to B$ be a mapping.
We say that $f$ \emph{preserves relations}, or that $f$ is a \emph{homomorphism} if:
\begin{itemize}
\item  $R^\calA(a_1, \ldots, a_n) \Rightarrow R^\calB(f(a_1), \ldots, f(a_n))$
      for all $R \in \Theta$ and $a_1, \ldots, a_n \in A$, where $n$ is the arity of~$R$.
\end{itemize}
We say that $f$ \emph{reflects relations} if:
\begin{itemize}
  \item for every $n \ge 1$, every $R \in \Theta$ of arity $n$ and all $b_1, \ldots, b_n \in \im(f)$,
        if $R^\calB(b_1, \ldots, b_n)$ then there are
        $a_1, \ldots, a_n \in A$ such that $R^\calA(a_1, \ldots, a_n)$ and $f(a_i) = b_i$, $1 \le i \le n$.
\end{itemize}
The mapping $f$ is an \emph{embedding} if it as an injective homomorphism which reflects relations.
Equivalently, $f$ is an embedding if and only if:
\begin{itemize}
\item $R^\calA(a_1, \ldots, a_n) \Leftrightarrow R^\calB(f(a_1), \ldots, f(a_n))$
      for all $R \in \Theta$ and $a_1, \ldots, a_n \in A$, where $n$ is the arity of~$R$.
\end{itemize}
Finally, $f$ is a \emph{quotient map} if it is a surjective homomorphism which reflects relations.
Let $\QMap(\calA, \calB)$ denote the set of all quotient maps $\calA \to \calB$.

Surjective embeddings are \emph{isomorphisms}. We write $\calA \cong \calB$ to denote that $\calA$ and $\calB$ are isomorphic.
A structure $\calA$ is a \emph{substructure} of a structure
$\calB$ ($\calA \le \calB$) if the identity map is an embedding of $\calA$ into $\calB$.
Let $\calA$ be a structure and $\0 \ne B \subseteq A$. Then $\calA[B] = (B, \restr {\Theta^\calA} B)$ denotes
the \emph{substructure of $\calA$ induced by~$B$}, where $\restr {\Theta^\calA} B$ denotes the restriction of
relations in $\Theta^\calA$ to~$B$.

\paragraph{Categories.}
In order to specify a \emph{category} $\CC$ one has to specify
a class of objects $\Ob(\CC)$, a class of morphisms $\hom_\CC(A, B)$ for all $A, B \in \Ob(\CC)$,
the identity morphism $\id_A$ for all $A \in \Ob(\CC)$, and
the composition of morphisms~$\cdot$~so that
$\id_B \cdot f = f = f \cdot \id_A$ for all $f \in \hom_\CC(A, B)$, and
$(f \cdot g) \cdot h = f \cdot (g \cdot h)$ whenever the compositions are defined.
A category $\CC$ is \emph{locally small} if $\hom_\CC(A, B)$ is a set for all $A, B \in \Ob(\CC)$.
Sets of the form $\hom_\CC(A, B)$ are then referred to as \emph{homsets}.
For $A, B \in \Ob(\CC)$ we write $A \to B$ if $\hom_\CC(A, B) \ne \0$.

For a category $\CC$, the \emph{opposite category}, denoted by $\CC^\op$, is the category whose objects
are the objects of $\CC$, morphisms are formally reversed so that
$
  \hom_{\CC^\op}(A, B) = \hom_\CC(B, A),
$
and so is the composition:
$
  f \cdot_{\CC^\op} g = g \cdot_\CC f.
$

The \emph{product} of categories $\CC_1$ and $\CC_2$ is the category $\CC_1 \times \CC_2$ whose objects are pairs $(A_1, A_2)$
where $A_1 \in \Ob(\CC_1)$ and $A_2 \in \Ob(\CC_2)$,
morphisms are pairs $(f_1, f_2) : (A_1, A_2) \to (B_1, B_2)$ where
$f_i$ is a morphism from $A_i$ to $B_i$ in $\CC_i$, $i \in \{1, 2\}$, and the composition of morphisms
is carried out componentwise: $(f_1, f_2) \cdot (g_1, g_2) = (f_1 \cdot g_1, f_2 \cdot g_2)$.

A category $\CC$ is \emph{directed} if for all $A, B \in \Ob(\CC)$ there is a $C \in \Ob(\CC)$ such that
$A \to C$ and $B \to C$. A category $\CC$ is \emph{dually directed} if $\CC^\op$ is directed.
A category $\CC$ has \emph{amalgamation} if for all $A, B, C \in \Ob(\CC)$ and morphisms $f \in \hom_\CC(A, B)$
and $g \in \hom_\CC(A, C)$ there is a $D \in \Ob(\CC)$ and morphisms $p \in \hom_\CC(B, D)$ and $q \in \hom_\CC(C, D)$ such that 
$p \cdot f = q \cdot g$. We say that $\CC$ has \emph{projective amalgamation} (or \emph{dual amalgamation}) if $\CC^\op$ has amalgamation.

A morphism $f \in \hom_\CC(B, C)$ is \emph{mono} or \emph{left cancellable} if
$f \cdot g = f \cdot h$ implies $g = h$ for all $g, h \in \hom_\CC(A, B)$ where $A \in \Ob(\CC)$ is arbitrary.
A morphism $f \in \hom_\CC(B, C)$ is \emph{epi} or \emph{right cancellable} if
$g \cdot f = h \cdot f$ implies $g = h$ for all $g, h \in \hom_\CC(C, D)$ where $D \in \Ob(\CC)$ is arbitrary.
A morphism $f \in \hom_\CC(A, B)$ is \emph{invertible} if there exists a morphism $g \in \hom_\CC(B, A)$
such that $g \cdot f = \id_A$ and $f \cdot g = \id_B$.
Two objects $A, B \in \Ob(\CC)$ are \emph{isomorphic}, and we write $A \cong B$,
if there exists an invertible morphism in $\hom_\CC(A, B)$.

A category $\DD$ is a \emph{subcategory} of a category $\CC$ if $\Ob(D) \subseteq \Ob(\CC)$ and $\hom_\DD(A, B) \subseteq \hom_\CC(A, B)$
for all $A, B \in \Ob(\DD)$. A subcategory $\DD$ of a category $\CC$ is a \emph{full subcategory} if
$\hom_\DD(A, B) = \hom_\CC(A, B)$ for all $A, B \in \Ob(\DD)$. Every class $\KK \subseteq \Ob(\CC)$ can be
turned into a category $\KK^*$ where $\Ob(\KK^*) = \KK$, $\hom_{\KK^*}(A, B) = \hom_\CC(A, B)$ for all
$A, B \in \KK$, and $f \cdot_{\KK^*} g = f \cdot_\CC g$. We say then that $\KK^*$ is a \emph{full subcategory of~$\CC$ spanned by~$\KK$.}

If $\CC$ is a category of structures, by $\CC^\fin$ we denote the full subcategory $\CC$ spanned by all finite
structures in $\CC$.

A \emph{functor} $F : \CC \to \DD$ from a category $\CC$ to a category $\DD$ maps $\Ob(\CC)$ to
$\Ob(\DD)$ and maps morphisms of $\CC$ to morphisms of $\DD$ so that
$F(f) \in \hom_\DD(F(A), F(B))$ whenever $f \in \hom_\CC(A, B)$, $F(f \cdot g) = F(f) \cdot F(g)$ whenever
$f \cdot g$ is defined, and $F(\id_A) = \id_{F(A)}$.
The identity functor $\ID_\CC : \CC \to \CC$ takes objects and morphisms of $\CC$ to themselves.
Categories $\CC$ and $\DD$ are isomorphic, in symbols $\CC \cong \DD$, if there exist functors
$F : \CC \to \DD$ and $G : \DD \to \CC$ such that $GF = \ID_\CC$ and $FG = \ID_\DD$.
A functor $F : \CC \to \DD$ is \emph{full} if it is surjective on homsets
(that is: for every $g \in \hom_\DD(F(A), F(B))$ there is an
$f \in \hom_\CC(A, B)$ with $F(f) = g$), and
\emph{faithful} if it is injective on homsets (that is: $F(f) = F(g)$ implies $f = g$).
A functor $U : \CC \to \DD$ is \emph{forgetful} if it is faithful and surjective on objects.
In this setting we may actually assume that $\hom_{\CC}(A, B) \subseteq \hom_\DD(U(A), U(B))$ for all $A, B \in \Ob(\CC)$.
The intuition behind this point of view is that $\CC$ is a category of structures, $\DD$ is the category of sets
and $U$ takes a structure $\calA$ to its underlying set $A$ (thus ``forgetting'' the structure). Then for every
morphism $f : \calA \to \calB$ in $\CC$ the same map is a morphism $f : A \to B$ in $\DD$.
Therefore, we shall always take that $U(f) = f$ for all the morphisms in $\CC$. In particular,
$U(\id_{A}) = \id_{U(A)}$ and we, therefore, identify $\id_{A}$ with $\id_{U(A)}$.
Also, if $U : \CC \to \DD$ is a forgetful functor and all the morphisms in $\DD$ are mono, then
all the morphisms in $\CC$ are mono.

For categories $\CC_1$ and $\CC_2$ there is a category $\CC_1 \times \CC_2$ whose objects are pairs $(A_1, A_2)$
where $A_1 \in \Ob(\CC_1)$ and $A_2 \in \Ob(\CC_2)$,
morphisms are pairs $(f_1, f_2) : (A_1, A_2) \to (B_1, B_2)$ where
$f_i$ is a morphism from $A_i$ to $B_i$ in $\CC_i$, $i \in \{1, 2\}$, and the composition of morphisms
is carried out componentwise: $(f_1, f_2) \cdot (g_1, g_2) = (f_1 \cdot g_1, f_2 \cdot g_2)$.
Clearly, if $\tilde A = (A_1, A_2)$ and $\tilde B = (B_1, B_2)$ are objects of $\CC_1 \times \CC_2$
then $\hom(\tilde A, \tilde B) = \hom(A_1, B_1) \times \hom(A_2, B_2)$.

\begin{EX}\label{drpperm.ex.CHrs-def}
  $(a)$ Let $\Set$ be the category of sets and all functions between them. By $\SetSurj$ we denote the category of sets and surjective functions,
  while $\Set^{\le\omega}$ denotes the category of finite and countably infinite sets and all functions. Analogously,
  $\SetSurj^{\le\omega}$ is the category of finite and countably infinite sets and surjective functions between them.

  $(b)$ By $\Top$ we denote the category of topological spaces and continuous maps.

  $(c)$ The composition of two rigid surjections is again a rigid surjection, so finite chains
  and rigid surjections constitute a category which we denote by~$\CHrs$.
  
  $(d)$ All well-ordered chains and rigid surjections constitute a category that we denote by $\WchRs$.
  By $\WchRs^{\le\omega}$ we denote the category of finite chains and countably infinite chains of order type $\omega$,
  together with rigid surjections between them.
\end{EX}

\paragraph{Finite Ramsey phenomena.}
Let $\CC$ be a locally small category. For integers $k, t \in \NN$ and objects $A, B, C \in \Ob(\CC)$
such that $A \to B \to C$ we write
$$
  C \longrightarrow (B)^{A}_{k, t}
$$
to denote that for every $k$-coloring $\chi : \hom(A, C) \to k$ there is a morphism
$w \in \hom(B, C)$ such that $|\chi(w \cdot \hom(A, B))| \le t$.
(For a set of morphisms $F$ we let $w \cdot F = \{ w \cdot f : f \in F \}$.)
Instead of $C \longrightarrow (B)^{A}_{k, 1}$ we simply write $C \longrightarrow (B)^{A}_{k}$.

For $A \in \Ob(\CC)$ let $t(A)$, the \emph{small Ramsey degree of $A$ in $\CC$},
denote the least positive integer $n$ such that
for all $k \in \NN$ and all $B \in \Ob(\CC)$ there exists a $C \in \Ob(\CC)$ such that
$C \longrightarrow (B)^{A}_{k, n}$, if such an integer exists.
Otherwise put $t(A) = \infty$.
A locally small category $\CC$ has the \emph{finite small Ramsey degrees} if $t(A) < \infty$
for all $A \in \Ob(\CC)$. Where necessary we shall write $t_\CC(A)$.
A category $\CC$ has the \emph{Ramsey property} if $t_\CC(A) = 1$ for all $A \in \Ob(\CC)$.

By a dual procedure we can straightforwardly introduce the notion of \emph{dual small Ramsey degrees}:
$$
  t^\partial_\CC(A) = t_{\CC^\op}(A).
$$
We then say that a locally small category $\CC$ has the
\emph{finite dual small Ramsey degrees} if $\CC^\op$ has finite small Ramsey degrees.
A category $\CC$ has the \emph{dual Ramsey property} if $t^\partial_\CC(A) = 1$ for all $A \in \Ob(\CC)$.

\begin{THM}[Finite Dual Ramsey Theorem \cite{graham-leeb-rothschild}]
  The category $\CHrs$ has the dual Ramsey property.
\end{THM}

Small Ramsey degrees are multiplicative in the following sense:

\begin{THM}[\cite{masul-kpt,masul-rpppg}]\label{rament.thm.srd-mult}\label{crt.thm.prod}
  Let $\CC_1$ and $\CC_2$ be categories whose morphisms are mono and homsets are finite. Then
  $t_{\CC_1 \times \CC_2}(A_1, A_2) = t_{\CC_1}(A_1) \cdot t_{\CC_2}(A_2)$, for all $A_1 \in \Ob(\CC_1)$ and $A_2 \in \Ob(\CC_2)$.
  This relationship holds even if $t_{\CC_1}(A_1) = \infty$ or $t_{\CC_2}(A_2) = \infty$ in which case
  $t_{\CC_1 \times \CC_2}(A_1, A_2) = \infty$.
\end{THM}

\begin{LEM}\label{rament.lem.APext}
  Let $\CC$ be a category with amalgamation whose homsets are finite.
  For all $A, B, C \in \Ob(\CC)$ and every morphism $f \in \hom(A, B)$ there is a $D \in \Ob(\CC)$
  and a morphism $w \in \hom(C, D)$ such that $w \cdot \hom(A, C) \subseteq \hom(B, D) \cdot f$.
\end{LEM}
\begin{proof}
  Let $\hom(A, C) = \{g_1, g_2, \ldots, g_n\}$. By the amalgamation property, for every $i \in \{1, \ldots, n\}$
  there is a $D_i \in \Ob(\CC)$ and morphisms $p_i \in \hom(C, D_i)$ and $q_i \in \hom(B, D_i)$ such that
  \begin{center}
    \begin{tikzcd}
      B \arrow[r, "q_i"] & D_i \\
      A \arrow[u, "f"] \arrow[r, "g_i"'] & C \arrow[u, "p_i"']
    \end{tikzcd}
  \end{center}
  Moreover, by the repeated use of the amalgamation property there is a $D \in \Ob(\CC)$ and morphisms
  $s_i \in \hom(D_i, D)$, $1 \le i \le n$, such that
  \begin{center}
    \begin{tikzcd}
         &    & D & \\
      D_1 \arrow[urr, "s_1"] & D_2 \arrow[ur, "s_2"'] & {\cdots} & D_n \arrow[ul, "s_n"'] \\
         &    & C \arrow[ull, "p_1"] \arrow[ul, "p_2"'] \arrow[ur, "p_n"']  & \\
    \end{tikzcd}
  \end{center}
  Let $w = s_1 \cdot p_1 = \ldots = s_n \cdot p_n \in \hom(C, D)$. To see that $D$ and $w$ satisfy the requirements of the
  lemma take any $g_i \in \hom(A, C)$ and note that
  $
    w \cdot g_i = s_i \cdot p_i \cdot g_i = s_i \cdot q_i \cdot f \in \hom(B, D) \cdot f.
  $
\end{proof}

\begin{THM} (cf.\ \cite{kechris-sokic-todorcevic})\label{drdrs.thm.cofinal-amalg}
  Let $\CC$ be a category with amalgamation and assume that $\hom(A, B)$ is finite for all $A, B \in \Ob(\CC)$.
  Let $\DD$ be a full subcategory of $\CC$ which is cofinal in $\CC$ and has small Ramsey degrees.
  Then $\CC$ has small Ramsey degrees.
\end{THM}
\begin{proof}
  Take any $k \in \NN$ and any $A, B \in \Ob(\CC)$ such that $A \to B$. Since $\DD$ is cofinal in $\CC$ there
  is a $P \in \Ob(\DD)$ and a morphism $a : A \to P$. Let $t = t_\DD(P)$. By Lemma~\ref{rament.lem.APext}
  there is a $Q \in \Ob(\DD)$ and a morphism $b : B \to Q$ such that
  $$
    b \cdot \hom(A, B) \subseteq \hom(P, Q) \cdot a.
  $$
  Since $t = t_\DD(P)$ there exists an $S \in \Ob(\DD)$ such that $S \longrightarrow (Q)^P_{k, t}$. Let us show that
  $S \longrightarrow (B)^A_{k, t}$. Take any $\chi : \hom(A, S) \to k$ and define $\chi' : \hom(P, S) \to k$ by
  $\chi'(f) = \chi(f \cdot a)$. By the choice of $S$ there is a $w : Q \to S$ such that
  $|\chi'(w \cdot \hom(P, Q))| \le t$. The fact that $|\chi(w \cdot b \cdot \hom(A, B))| \le t$
  follows from:
  $$
    \chi(w \cdot b \cdot \hom(A, B)) \subseteq \chi(w \cdot \hom(P, Q) \cdot a) \subseteq \chi'(w \cdot \hom(P, Q)).
  $$
  This completes the proof.
\end{proof}

We shall also need \cite[Theorem~7]{masul-drp}, but in a more elaborate form.
Consider a finite, acyclic, bipartite digraph with loops 
where all the arrows go from one class of vertices into the other
and the out-degree of all the vertices in the first class is~2 (modulo loops):
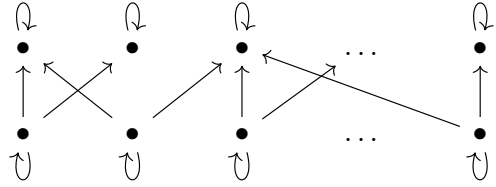
\begin{center}
  \begin{tikzcd}
    {\bullet} \arrow[loop above] & {\bullet} \arrow[loop above] & {\bullet} \arrow[loop above] & \ldots & {\bullet} \arrow[loop above] \\
    {\bullet} \arrow[loop below] \arrow[u] \arrow[ur] & {\bullet} \arrow[loop below] \arrow[ur] \arrow[ul] & {\bullet} \arrow[loop below] \arrow[u] \arrow[ur] & \ldots & {\bullet} \arrow[loop below] \arrow[u] \arrow[ull]
  \end{tikzcd}
\end{center}
Such a diagraph can be thought of as a category where the loops represent the identity morphisms, and will be referred to as
a \emph{binary category}. (Note that all the compositions in a binary category
are trivial since no nonidentity morphisms are composable.)

An \emph{amalgamation problem} in a category $\CC$ is a diagram $F : \Delta \to \CC$ where $\Delta$ is a binary category,
$F$ takes the top row of $\Delta$ to the same object, and takes the bottom of $\Delta$ to the same object,
see Fig.~\ref{nrt.fig.2}. If $F$ takes the bottom row of $\Delta$ to an object $A$ and the top
row to an object $B$ then the diagram $F : \Delta \to \CC$ will be referred to as the \emph{$(A, B)$-diagram in $\CC$}.
An amalgamation problem $F : \Delta \to \CC$ \emph{has a solution in $\CC$} if $F$ has a compatible cocone
in $\CC$. An amalgamation problem $F : \Delta \to \CC$ is \emph{finite} if $\Delta$ is a finite category.
\begin{figure}
  \centering
  \begin{tikzcd}
    {\bullet} & {\bullet} & {\bullet}
    & & B & B & B
  \\
    {\bullet} \arrow[u] \arrow[ur] & {\bullet} \arrow[ur] \arrow[ul] & {\bullet} \arrow[ul] \arrow[u]
    & & A \arrow[u] \arrow[ur] & A \arrow[ur] \arrow[ul] & A \arrow[ul] \arrow[u]
  \\
    & \Delta \arrow[rrrr, "F"]  & & & & \CC  
  \end{tikzcd}
  \caption{An $(A,B)$-diagram in $\CC$ (of shape $\Delta$)}
  \label{nrt.fig.2}
\end{figure}
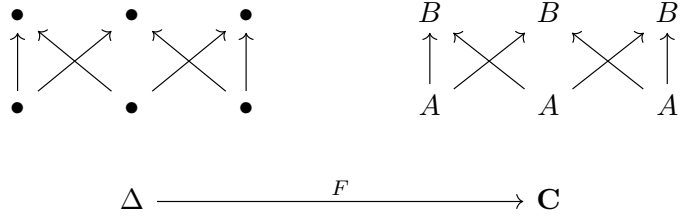
  
Let $\BB$ and $\CC$ be locally small categories such that morphisms in both $\BB$ and $\CC$ are mono
and let $G : \BB \to \CC$ be a faithful functor. For $A, B \in \Ob(\BB)$ and $C \in \Ob(\CC)$ let us
construct a binary category $\Delta_G(A, B; C)$ as follows. Let
$$
  \hom_\CC(G(B), C) = \{e_i : i \in I\}.
$$
Intuitively, for each $e_i \in \hom_\CC(G(B), C)$ we add a copy of $B$ to the category, and whenever $e_i \cdot G(u) = e_j \cdot G(v)$
in $\CC$ for some $u, v \in \hom_\BB(A, B)$ we add a copy of A to the category together with two arrows -- one going into the $i$th copy of $B$
labelled by $u$ and another one going into the $j$th copy of $B$ labelled by~$v$:
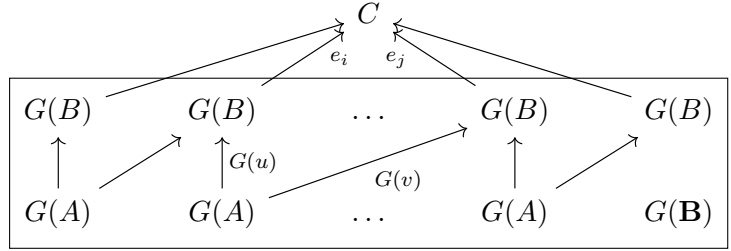
\begin{center}
  \begin{tikzcd}[execute at end picture={
                     \draw (-4.75,-1.75) rectangle (4.75,0.5);
                }]
    & & C
  \\
    G(B) \arrow[urr] & G(B) \arrow[ur, near end, "e_i"'] & \ldots & G(B) \arrow[ul, near end, "e_j"] & G(B) \arrow[ull]
  \\
    G(A) \arrow[u] \arrow[ur] & G(A) \arrow[urr, "G(v)"'] \arrow[u, "G(u)"'] & \ldots & G(A) \arrow[ur] \arrow[u] & G(\BB)
  \end{tikzcd}
\end{center}
Formally, let $\Delta = \Delta_G(A, B; C)$ be the binary category whose objects are
$$
  \Ob(\Delta) = I \union \{(u, v, i, j) : i, j \in I; 
                                          u, v \in \hom_\BB(A, B); e_i \cdot G(u) = e_j \cdot G(v)\}
$$
and whose nonidentity arrows are
$$
  \hom_\Delta((u, v, i, j), i) = \{u\} \text{\quad and\quad} \hom_\Delta((u, v, i, j), j) = \{v\}.
$$

\begin{PROP}\label{crt.prop.D-solves-Delta-diag} (cf.~\cite[Theorem~7]{masul-drp})
  Let $\BB$ and $\CC$ be locally small categories such that morphisms in both $\BB$ and $\CC$ are mono
  and let $G : \BB \to \CC$ be a faithful functor. Take any $A, B \in \Ob(\BB)$ and $C \in \Ob(\CC)$ and let
  $F : \Delta_G(A, B; C) \to \BB$ be the diagram in $\BB$ whose action on objects is:
  $$
    F(s) = B \text{\quad and\quad} F((u, v, i, j)) = A
  $$
  for all $s, (u,v,i,j) \in \Ob(\Delta_G(A, B; C))$, and whose action on nonidentity morphisms is $F(g) = g$:
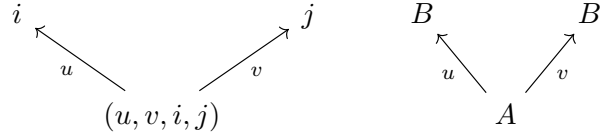
  \begin{center}
    \begin{tikzcd}[column sep=small]
    i &  & j
    & & B &  & B
  \\
      & (u, v, i, j) \arrow[ur, "v"'] \arrow[ul, "u"] &  
    & &   & A \arrow[ur, "v"'] \arrow[ul, "u"] & 
  \\
    & \Delta_G(A, B; C) \arrow[rrrr, "F"]  & & & & \BB
    \end{tikzcd}
  \end{center}
  If $C \longrightarrow (G(B))^{G(A)}_{k,t}$ and $D \in \Ob(\BB)$ together with morphisms $f_i \in \hom_\BB(B, D)$, $i \in I$,
  solves $F : \Delta_G(A, B; C) \to \BB$ so that the corresponding diagram in $\BB$ commutes:
  
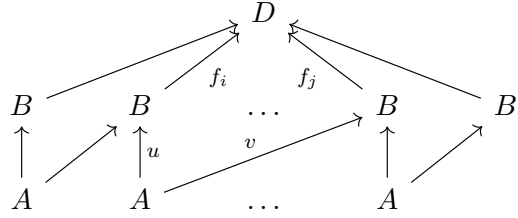
\begin{center}
    \begin{tikzcd}
        & & D
      \\
        B \arrow[urr] & B \arrow[ur, "f_i"'] & \ldots & B \arrow[ul,"f_j"] & B \arrow[ull]
      \\
        A \arrow[u] \arrow[ur] & A \arrow[urr,"v"] \arrow[u, "u"'] & \ldots & A \arrow[ur] \arrow[u] & 
    \end{tikzcd}
  \end{center}
  then $D \longrightarrow (B)^A_{k, t}$ in $\DD$.
  In particular, $t_\BB(A) \le t_\CC(G(A))$ for all $A \in \Ob(\BB)$.
\end{PROP}

\paragraph{Expansions.}
An \emph{expansion} of a category $\CC$ is a category $\CC^*$ together with a faithful functor
$U : \CC^* \to \CC$ which is surjective on objects.
We shall generally follow the convention that $A, B, C, \ldots$ denote objects from $\CC$
while $A^*, B^*, C^*, \ldots$ denote objects from $\CC^*$.
Since $U$ is injective on hom-sets we may safely assume that
$\hom_{\CC^*}(A^*, B^*) \subseteq \hom_\CC(A, B)$ where $A = U(A^*)$, $B = U(B^*)$.
In particular, $\id_{A^*} = \id_A$ for $A = U(A^*)$. Moreover, it is safe to drop
subscripts $\CC$ and $\CC^*$ in $\hom_\CC(A, B)$ and $\hom_{\CC^*}(A^*, B^*)$, so we shall
simply write $\hom(A, B)$ and $\hom(A^*, B^*)$, respectively.
Let $U^{-1}(A) = \{A^* \in \Ob(\CC^*) : U(A^*) = A \}$. Note that this is not necessarily a set.

An expansion $U : \CC^* \to \CC$ is \emph{reasonable} (cf.~\cite{KPT}) if
for every $e \in \hom(A, B)$ and every $A^* \in U^{-1}(A)$ there is a $B^* \in U^{-1}(B)$ such that
$e \in \hom(A^*, B^*)$:
\begin{center}
    \begin{tikzcd}
      A^* \arrow[r, "e"] \arrow[d, dashed, mapsto, "U"'] & B^* \arrow[d, dashed, mapsto, "U"] \\
      A \arrow[r, "e"'] & B
    \end{tikzcd}
\end{center}
An expansion $U : \CC^* \to \CC$ is \emph{dually reasonable} if $U^\op : (\CC^*)^\op \to \CC^\op$ is reasonable.

An expansion $U : \CC^* \to \CC$ has \emph{restrictions} if
for every $B^* \in \Ob(\CC^*)$ and every $e \in \hom(A, U(B^*))$ there is an $A^* \in U^{-1}(A)$
such that $e \in \hom(A^*, B^*)$:
\begin{center}
    \begin{tikzcd}
      A^* \arrow[r, "e"] \arrow[d, dashed, mapsto, "U"'] & B^* \arrow[d, dashed, mapsto, "U"] \\
      A \arrow[r, "e"'] & B
    \end{tikzcd}
\end{center}
If such an $A^* \in U^{-1}(A)$ is unique, we say that the expansion $U : \CC^* \to \CC$ has \emph{unique restrictions}.
We denote this unique $A^*$ by $\restr {B^*} e$ and refer to it as \emph{the restriction of $B^*$ along~$e$}.
An expansion $U : \CC^* \to \CC$ has \emph{(unique) dual restrictions} if $U^\op : (\CC^*)^\op \to \CC^\op$
has (unique) restrictions.

Let $\CC$ and $\CC^*$ be categories of structures.
An expansion $U : \CC^* \to \CC$ is \emph{precompact} (cf.~\cite{vanthe-more}) if
$U^{-1}(A)$ is a set for all $A \in \Ob(\CC)$, and it is a finite set for all finite $A \in \Ob(\CC)$.
(See~\cite{masul-kpt} for a general notion that applies to arbitrary categories.)
Following \cite{vanthe-more} we say that $U : \CC^* \to \CC$ \emph{has the expansion property}
if for every $A \in \Ob(\CC)$ there exists a $B \in \Ob(\CC)$ such that $A^* \to B^*$ for all
$A^*, B^* \in \Ob(\CC^*)$ with $U(A^*) = A$ and $U(B^*) = B$.
We say that $U : \CC^* \to \CC$ \emph{has the dual expansion property}
if $U^\op : (\CC^*)^\op \to \CC^\op$ has the expansion property.

\begin{THM}\label{sbrd.thm.small1}\label{sbrd.thm.small2}\cite{dasilvabarbosa,masul-kpt}
  $(a)$ Let $U : \CC^* \to \CC$ be a reasonable expansion with restrictions and assume that all the morphisms in $\CC$
  are mono. For any $A \in \Ob(\CC)$ we then have:
  $$
    t_{\CC}(A) \le \sum_{A^* \in U^{-1}(A)} t_{\CC^*}(A^*).
  $$
  Consequently, if $U^{-1}(A)$ is finite and
  $t_{\CC^*}(A^*) < \infty$ for all $A^* \in U^{-1}(A)$ then $t_\CC(A) < \infty$.

  $(b)$ If in addition to that we also assume that $\CC^*$ is directed and that $U : \CC^* \to \CC$ is an expansion with unique restrictions
  and with the expansion property then for any $A \in \Ob(\CC)$:
  $$
    t_{\CC}(A) = \sum_{A^* \in U^{-1}(A)} t_{\CC^*}(A^*).
  $$
\end{THM}

\begin{PROP}\label{drdrs.prop.t-SetFin-n}
  For every integer $n \ge 1$ the category $\FinSetSurj^n$ has dual small Ramsey degrees.
\end{PROP}
\begin{proof}
  It suffices to show that $\FinSetSurj$ has dual small Ramsey degrees because the statement then follows from
  the dual of Theorem~\ref{rament.thm.srd-mult}. To this end, let $U : \CHrs \to \FinSetSurj$ be the forgetful functor
  which forgets the ordering. It is easy to check that $U$ is a reasonable expansion with restrictions, so
  for every nonempty finite set $A$ Theorem~\ref{sbrd.thm.small1} ensures that $t^\partial(A) < \infty$ in $\FinSetSurj$.
  Note that $t^\partial(\0) = 1$ in $\FinSetSurj$ by definition. This completes the proof.
\end{proof}

\paragraph{Infinite Ramsey phenomena.}
The Infinite Dual Ramsey Theorem of Carlson and Simpson \cite{carlson-simpson-1984}
is out main tool for establishing existence of dual big Ramsey degrees, but its presentation
requires additional infrastructure.

A category $\CC$ is \emph{enriched over $\Top$} if each $\hom_\CC(A, B)$ is a topological space
and the composition $\Boxed\cdot : \hom_\CC(B, C) \times \hom_\CC(A, B) \to \hom_\CC(A, C)$ is continuous
for all $A, B, C \in \Ob(\CC)$.
(Note that a category enriched over $\Top$ has to be locally small.)

\begin{EX}
  $(a)$ Any locally small category can be thought of as a category enriched over $\Top$ by taking each
  hom-set to be a discrete space. We shall refer to this as the \emph{discrete enrichment}.
  
  $(b)$ The category $\Set$ can be enriched over $\Top$ in a nontrivial way by endowing
  each hom-set $\hom_{\Set}(A, B)$ with the pointwise convergence topology.
  Whenever we refer to $\Set$ as a category enriched over $\Top$ we have this particular
  enrichment in mind. The same applies to $\SetSurj$, $\Set^{\le \omega}$ and $\SetSurj^{\le\omega}$.
  
  $(c)$ The category $\WchRs$ can be enriched over $\Top$ in a nontrivial way as follows:
  each hom-set $\hom_{\WchRs}((A, \Boxed<), (B, \Boxed<))$ inherits the topology from the Tychonoff topology on $B^A$
  with both $A$ and $B$ discrete. Whenever we refer to $\WchRs$ as a category enriched over $\Top$ we have this particular
  enrichment in mind. The same applies to $\WchRs^{\le \omega}$.
\end{EX}

For a topological space $X$ and an integer $k \ge 2$ a \emph{Borel $k$-coloring of $X$} is any mapping $\chi : X \to k$
such that $\chi^{-1}(i)$ is a Borel set for all~$i \in k$.

Let $\CC$ be a category enriched over $\Top$.
For $A, B, C \in \Ob(\CC)$ we write $C \Borel\longrightarrow (B)^{A}_{k, n}$
to denote that for every Borel $k$-coloring $\chi : \hom_\CC(A, C) \to k$
there is a morphism $w \in \hom_\CC(B, C)$ such that $|\chi(w \cdot \hom_\CC(A, B))| \le n$.
For $A, S \in \Ob(\CC)$ let $T_\CC(A, S)$ denote the least positive integer $n$ such that
$S \Borel\longrightarrow (S)^{A}_{k, n}$ for all $k \ge 2$, if such an integer exists.
Otherwise put $T_\CC(A, S) = \infty$. The number $T_\CC(A, S)$ will be referred to as the
\emph{big Ramsey degree of $A$ in $S$ with respect to Borel colorings}.
By straightforward dualization, we can introduce \emph{big dual Ramsey degrees with respect to Borel colorings}
$T^\partial_\CC(A, S)$ by $T^\partial_\CC(A, S) = T_{\CC^\op}(A, S)$.
The Infinite Dual Ramsey Theorem of Carlson and Simpson~\cite{carlson-simpson-1984}
now takes the following form.

\begin{THM}[The Infinite Dual Ramsey Theorem~\cite{carlson-simpson-1984}]\label{rpemklei.thm.IDRT}
  In the category $\WchRs$ enriched over $\Top$ as above,
  $T^\partial(A, \omega) = 1$ for every finite chain~$A$.
\end{THM}

Let $\CC$ be a category of relational structures over the same relational language with quotient maps as morphisms.
We say that an $\calS \in \Ob(\CC)$ is \emph{projectively universal} for the category $\CC^\fin$ of all finite
structures in $\CC$ if $\hom_\CC(\calS, \calA) \ne \0$ for all $\calA \in \Ob(\CC^\fin)$.
We say that $\calC, \calD \in \Ob(\CC)$ are \emph{bi-projectible} if there is a quotient map
$\calC \to \calD$ and a quotient map $\calD \to \calC$. This notion is relevant because of the
following simple but useful fact:

\begin{LEM} \label{drdrs.lem.bi-proj-general}
  Let $\CC$ be a category of relational structures over the same relational language with quotient maps as morphisms.
  If $\calC, \calD \in \Ob(\CC)$ are bi-projectible then
  $T_\CC^\partial(\calA, \calC) = T_\CC^\partial(\calA, \calD)$ for every finite $\calA \in \Ob(\CC)$.
\end{LEM}

Let $\CC$ be a category and $A, B, S \in \Ob(\CC)$. We say that
$S$ \emph{is weakly homogeneous for $(A, B)$}, if there exist
$f \in \hom(A, B)$ and $g \in \hom(S, S)$ such that $g \cdot \hom(A, S) \subseteq \hom(B, S) \cdot f$.
\begin{center}
  \begin{tikzcd}
    B \arrow[rr, bend left=20] \arrow[rr, bend right=20] \arrow[ddrr, phantom, "\supseteq" description] & & S \\
    & \\
    A \arrow[uu, "f"] \arrow[rr, bend left=20] \arrow[rr, bend right=20] & & S \arrow[uu, "g"']
  \end{tikzcd}
\end{center}

\begin{THM}\label{rdbas.thm.mon-bigdeg}
  Let $\CC$ be a category with enriched over $\Top$ whose morphisms are mono and let $A, B \in \Ob(\CC)$ be such that $A \to B$.
  Then $T(A, S) \le T(B, S)$ for every $S \in \Ob(\CC)$ which is weakly homogeneous for $(A, B)$.
\end{THM}
\begin{proof}
  If $T(B, S) = \infty$ the statement is trivially true, so assume that $T(B, S) = n \in \NN$.
  Take any $S$ which is weakly homogeneous for $(A, B)$. Then there exist
  $f \in \hom(A, B)$ and $g \in \hom(S, S)$ such that $g \cdot \hom(A, S) \subseteq \hom(B, S) \cdot f$.
  Take any $k \in \NN$ and let $\chi : \hom(A, S) \to k$ be a Borel coloring.
  Define $\chi' : \hom(B, S) \to k$ by $\chi'(h) = \chi(h \cdot f)$.
  Note that this is a Borel coloring.
  Then there is a $w \in \hom(S, S)$ such that
  $|\chi'(w \cdot \hom(B, S))| \le n$. The definition of $\chi'$ then yields
  $|\chi(w \cdot \hom(B, S) \cdot f)| \le n$. Therefore,
  $|\chi(w \cdot g \cdot \hom(A, S))| \le n$ because $g \cdot \hom(A, S) \subseteq \hom(B, S) \cdot f$.
\end{proof}

Let $U : \CC^* \to \CC$ be an expansion with unique restrictions.
We say that $S^* \in \Ob(\CC^*)$ is \emph{self-similar with respect to $U$} \cite{masul-rdbas} if the following holds:
for every $w \in \hom_\CC(S, S)$ there is a morphism $v \in \hom_{\CC^*}(S^*, \restr{S^*}{w})$,
where $S = U(S^*)$.
\begin{center}
  \begin{tikzcd}
    S^* \arrow[r, "\exists v"] & \restr{S^*}{w} \arrow[r, "w"] \arrow[d, mapsto, dashed, "U"'] & S^* \arrow[d, mapsto, dashed, "U"]\\
       & S \arrow[r, "\forall w"] & S
  \end{tikzcd}
\end{center}
We say that $S^* \in \Ob(\CC^*)$ is \emph{dually self-similar with respect to $U$} is $S^*$ is self-similar with respect to
$U^\op : (\CC^*)^\op \to \CC^\op$.

\begin{THM}\label{sbrd.thm.big1-dual-borel-weaker} (cf.~\cite{masul-drp-algs})
  Let $\CC$ and $\CC^*$ be categories enriched over $\Top$ such that morphisms in both $\CC$ and $\CC^*$
  are mono.
  Let $U : \CC^* \to \CC$ be an expansion with restrictions and assume that $\hom_{\CC^*}(A^*, B^*)$ is
  a Borel set in $\hom_{\CC}(U(A^*), U(B^*))$ for all $A^*, B^* \in \Ob(\CC^*)$.
  For $A \in \Ob(\CC)$, $S^* \in \Ob(\CC^*)$ and $S = U(S^*)$, if $U^{-1}(A)$ is finite then:
  $$
    T_\CC(A, S) \le \sum_{A^* \in U^{-1}(A)} T_{\CC^*}(A^*, S^*).
  $$
  If in addition to that we also assume that $U : \CC^* \to \CC$ is an expansion with unique restrictions and that
  $S^*$ is self-similar with respect to $U$ then
  $$
    T_\CC(A, S) = \sum_{A^* \in U^{-1}(A)} T_{\CC^*}(A^*, S^*).
  $$
\end{THM}
\begin{proof}
  The proof is analogous to the proof of the corresponding statements in \cite{masul-rdbas}. The only
  difference is that we now have to ensure that all the colorings we construct along the way are
  Borel. But that follows immediately from the assumption that $\hom_{\CC^*}(A^*, B^*)$ is
  a Borel set in $\hom_{\CC}(U(A^*), U(B^*))$ for all $A^*, B^* \in \Ob(\CC^*)$ and the fact that
  the composition of morphisms is continuous.
\end{proof}

Fix an expansion $U : \CC^* \to \CC$ and for each pair of objects $A^*, B^* \in \Ob(\CC^*)$ let
$U_{A^*,B^*}$ denote the set-function
$$
  U_{A^*,B^*} : \hom_{\CC^*}(A^*, B^*) \to \hom_{\CC}(U(A^*), U(B^*))
$$
between the corresponding hom-sets. Since $U$ is an expansion, every $U_{A^*, B^*}$ is injective,
so by the Lusin-Suslin Theorem, in order to show that $U$ preserves Borel sets it suffices to show that
each $U_{A^*, B^*}$ is continuous.

For sets $A$ and $B$ let $\Map(A, B) = B^A$ denote the set of all the functions $A \to B$, and let $\Surj(A, B)$
denote the set of all the surjective functions $A \to B$.
  
\begin{PROP}\label{drdrs.prop.T-Set-n}
  $(a)$ $T^\partial_{\SetSurj}(A, \omega) < \infty$ for every nonempty finite set $A$.

  \medskip

  $(b)$ For every nonempty finite set $A$ there exists a positive integer $t$ such that for every finite Borel
  coloring $\chi : \Map(\omega, A) \to k$ there exists a surjective map $r : \omega \to \omega$ such that
  $|\chi(\Map(\omega, A) \circ r)| \le t$.

  \medskip

  $(c)$ For every integer $n \ge 1$ and nonempty finite sets $A_1$, \ldots, $A_n$ we have
  $
    T^\partial_{\SetSurj^n}\big((A_1, \ldots, A_n), (\omega, \ldots, \omega)\big) < \infty,
  $
  where $\SetSurj^n = \SetSurj \times \ldots \times \SetSurj$ ($n$ times).
\end{PROP}
\begin{proof}
  $(a)$ Let $A$ be a nonempty finite set and let $\CC$ be the full subcategory of $\SetSurj^{\le\omega}$ spanned by
  nonempty sets. It is clear that $T^\partial_{\SetSurj}(A, \omega) = T^\partial_{\CC}(A, \omega)$ for every
  nonempty finite set $A$, so let us show that $T^\partial_{\CC}(A, \omega) < \infty$.
  Let $U : \WchRs^{\le\omega} \to \CC$ be the forgetful functor that forgets the ordering.
  It is easy to see that $U$ is an expansion with restrictions, and that $U_{A^*,B^*}$ is continuous for all $A^*, B^* \in \Ob(\WchRs^{\le \omega})$.
  Moreover, for every nonempty finite set $A$ we have that $|U^{-1}(A)|$ is finite.
  So, if $A$ is a nonempty finite set then $T^\partial_{\CC}(A, \omega) < \infty$ by Theorem~\ref{sbrd.thm.big1-dual-borel-weaker}.

  $(b)$ Let $A$ be a nonempty finite set. Let us enumerate all the nonempty subsets of $A$ as
  $B_1$, $B_2$, \ldots, $B_m$ and let $t_i = T^\partial_{\SetSurj}(B_i, \omega)$, $1 \le i \le m$. Note that by $(a)$
  all the $t_i$'s are positive integers, so $t = t_1 + \ldots + t_m$ is also a positive integer.
  
  Let $\chi : \Map(\omega, A) \to k$ be a finite Borel coloring. Clearly, $\Map(\omega, A) = \bigcup_{i=1}^m \Surj(\omega, B_i)$ and this is a disjoint union.
  Since $\Surj(\omega, B_i)$ is a Borel subset of $\Map(\omega, A)$ all the restrictions
  $$
    \chi_i : \Surj(\omega, B_i) \to k : f \mapsto \chi(f)
  $$
  of $\chi$ are Borel colorings, $1 \le i \le m$. Let us inductively construct a sequence $q_1, q_2, \ldots, q_m$ of surjections
  $\omega \to \omega$ as follows.

  To start the induction, recall that $\chi_1 : \Surj(\omega, B_1) \to k$ is a Borel coloring. 
  Since $t_1 = T^\partial_{\SetSurj}(B_1, \omega)$, there is a surjection $q_1 : \omega \to \omega$ such that
  $$
    |\chi_1(\Surj(\omega, B_1) \circ q_1)| = t_1.
  $$
  For the induction step assume that $q_1, \ldots, q_{j-1}$ have been constructed
  and consider the Borel coloring $\gamma_j : \Surj(\omega, B_j) \to k$ given by
  $$
    \gamma_j(f) = \chi_j(f \circ q_{j-1} \circ \ldots \circ q_1).
  $$
  Since $t_j = T^\partial_{\SetSurj}(B_j, \omega)$, there is a surjection $q_j : \omega \to \omega$ such that
  $$
    |\gamma_j(\Surj(\omega, B_j) \circ q_j)| = t_j.
  $$

  Let $r = q_m \circ q_{m-1} \circ \ldots \circ q_1$. Then $r$ is clearly a surjection and
  $$
    \chi(\Map(\omega, A) \circ r) = \bigcup\nolimits_{j=1}^m \chi_j(\Surj(\omega, B_j) \circ r).
  $$
  It is easy to see that for all $1 \le j \le m$:
  $$
    \Surj(\omega, B_j) \circ r \subseteq \Surj(\omega, B_j) \circ q_j \circ q_{j-1} \circ \ldots \circ q_1,
  $$
  whence
  \begin{multline*}
    \bigcup\nolimits_{j=1}^m \chi_j(\Surj(\omega, B_j) \circ r) \subseteq\\
    \subseteq \bigcup\nolimits_{j=1}^m \chi_j(\Surj(\omega, B_j) \circ q_j \circ q_{j-1} \circ \ldots \circ q_1) =\\
    = \bigcup\nolimits_{j=1}^m \gamma_j(\Surj(\omega, B_j) \circ q_j),
  \end{multline*}
  where, to keep the notation simple, we let $\gamma_1 = \chi_1$. Finally,
  $$
    |\chi(\Map(\omega, A) \circ r)| \le \sum\nolimits_{j=1}^m |\gamma_j(\Surj(\omega, B_j) \circ q_j)| = \sum\nolimits_{j=1}^m t_j = t,
  $$
  which concludes the proof of $(b)$.

  $(c)$
  For $f_1 \in \Map(B, A_1)$, \ldots, $f_n \in \Map(B, A_n)$ by $\langle f_1, \ldots, f_n\rangle$ we denote a function
  in $\Map(B, A_1 \times \ldots \times A_n)$ given by
  $$
    \langle f_1, \ldots, f_n \rangle(b) = (f_1(b), \ldots, f_n(b)).
  $$
  Fix an integer $n \ge 1$ and nonempty finite sets $A_1$, \ldots, $A_n$, and
  let $t$ be the positive integer provided by $(b)$ for $A_1 \times \ldots \times A_n$.
  Take any finite Borel coloring
  $$
    \chi : \hom_{\SetSurj^n}\big((\omega, \ldots, \omega), (A_1, \ldots, A_n)\big) \to k.
  $$
  Define $\gamma : \Map(\omega, A_1 \times \ldots \times A_n) \to k$
  so that for $f_1 \in \Surj(\omega, A_1)$, \ldots, $f_n \in \Surj(\omega, A_n)$ we have
  $$
    \gamma\big(\langle f_1, \ldots, f_n\rangle\big) = \chi(f_1, \ldots, f_n),
  $$
  and $\gamma(g) = 0$ otherwise. The coloring $\gamma$ is Borel because $\chi$ is a Borel coloring and the map
  $$
    \langle -, \ldots, -\rangle : \Surj(\omega, A_1) \times \ldots \times \Surj(\omega, A_n) \to \Map(\omega, A_1 \times \ldots \times A_n)
  $$
  is continuous and injective. Here we use the fact that
  $$
    \hom_{\SetSurj^n}\big((\omega, \ldots, \omega), (A_1, \ldots, A_n)\big) = \Surj(\omega, A_1) \times \ldots \times \Surj(\omega, A_n).
  $$
  According to $(b)$ there exists a surjective map $r : \omega \to \omega$ such that
  $$
    |\gamma(\Map(\omega, A_1 \times \ldots \times A_n) \circ r)| \le t.
  $$
  Let $\widehat r = (r, \ldots, r) \in \hom_{\SetSurj^n}\big((\omega, \ldots, \omega), (\omega, \ldots, \omega)\big)$ and let us show
  that
  $$
    |\chi(\hom_{\SetSurj^n}\big((\omega, \ldots, \omega), (A_1, \ldots, A_n)\big) \circ \widehat r)| \le t
  $$
  by showing that
  \begin{multline*}
    \chi(\hom_{\SetSurj^n}\big((\omega, \ldots, \omega), (A_1, \ldots, A_n)\big) \circ \widehat r) \subseteq\\
      \subseteq \gamma(\Map(\omega, A_1 \times \ldots \times A_n) \circ r).
  \end{multline*}
  Take any $(f_1, \ldots, f_n) \in \hom_{\SetSurj^n}\big((\omega, \ldots, \omega), (A_1, \ldots, A_n)\big)$. Then
  \begin{multline*}
    \chi\big((f_1, \ldots, f_n) \circ \widehat r\big) = \chi(f_1 \circ r, \ldots, f_n \circ r) =\\
    = \gamma(\langle f_1 \circ r, \ldots, f_n \circ r \rangle) = \gamma(\langle f_1, \ldots, f_n\rangle  \circ r)
  \end{multline*}
  This completes the proof.
\end{proof}

\section{Relational structures over finite languages}
\label{drdrs.sec.rsfl}

Let $\Theta$ be a relational language, let $\RelQuo(\Theta)$ denote the category of $\Theta$-structures
and quotient maps, and let $\FinRelQuo(\Theta)$ denote the category of finite $\Theta$-structures
and quotient maps. Assuming that $\Theta$ is a finite relational language, in this section we show that
$\FinRelQuo(\Theta)$ has dual small Ramsey degrees, and that in $\RelQuo(\Theta)$ every finite $\Theta$-structure
has a finite dual big Ramsey degree in any projectively universal countable $\Theta$-structure.

Let $\Theta = \{R_1, R_2, \ldots, R_m\}$ be a finite relational language where $R_i$ is a relational symbol of arity $h_i \ge 1$,
$1 \le i \le m$. The presence of isolated points turns out to be a major technical issue, so our strategy will be to first
consider $\Theta$-structures with no isolated points, and then infer the general result from that.
To this end, let $\RelQuo^+(\Theta)$ denote the full subcategory of $\RelQuo(\Theta)$ spanned by finite $\Theta$-structures with no isolated points,
and let $\FinRelQuo^+(\Theta)$ denote the full subcategory of $\RelQuo^+(\Theta)$ spanned by finite structures.

Our proofs will rely on analyzing the Ramsey properties for the the following very special $\Theta$-structures.
For arbitrary sets $X_1, \ldots, X_m, Y$ let $\calE^Y_{X_1, \ldots, X_m} = (E^Y_{X_1, \ldots, X_m}, \Theta^{E^Y_{X_1, \ldots, X_m}})$
be the $\Theta$-structure with
$$
  E^Y_{X_1, \ldots, X_m} = (Y \times \{0\}) \cup \bigcup_{i=1}^m X_i \times \{1, \ldots, h_i\} \times \{i\}
$$
and for $1 \le i \le m$:
$$
  R_i^{E^Y_{X_1, \ldots, X_m}} = \{ \big((x, 1, i), \ldots, (x, h_i, i)\big) : x \in X_i \}.
$$
Intuitively, $Y$ is the set of isolated points of $\calE^Y_{X_1, \ldots, X_m}$ and
each $R_i^{E^Y_{X_1, \ldots, X_m}}$ consists of $|X_i|$ many independent $h_i$-tuples.
Let $\EE(\Theta)$ be the category whose objects are $\calE^Y_{X_1, \ldots, X_m}$ where
$X_1, \ldots, X_m, Y$ are arbitrary sets and whose morphisms are quotient maps, and let
$\EE^+(\Theta)$ denote the full subcategory of $\EE(\Theta)$ spanned by structures
$\calE^\0_{X_1, \ldots, X_m}$ with no isolated points. Moreover, let
$\FinE(\Theta)$ denote the full subcategory of $\EE(\Theta)$ spanned by finite structures,
and let $\FinE^+(\Theta)$ denote the full subcategory of $\FinE(\Theta)$ spanned by structures
$\calE^\0_{X_1, \ldots, X_m}$ with no isolated points.

\begin{LEM}\label{drdrs.lem.E-iso-Set-m}
  Let $\Theta = \{R_1, \ldots, R_m\}$ be a finite relational language.
  
  $(a)$ The categories $\EE^+(\Theta)$ and $\SetSurj^m$ are isomorphic.

  $(b)$ The categories $\FinE^+(\Theta)$ and $\FinSetSurj^m$ are isomorphic.
\end{LEM}
\begin{proof}
  $(a)$ Define $F : \SetSurj^m \to \EE^+(\Theta)$ as follows.
  The action of $F$ on objects is given by
  $
    F(X_1, \ldots, X_m) = \calE^\0_{X_1, \ldots, X_m}
  $,
  while for $f_i : X_i \to Y_i$, $1 \le i \le m$, we define
  $
    F(f_1, \ldots, f_m) : \calE^\0_{X_1, \ldots, X_m} \to \calE^\0_{Y_1, \ldots, Y_m}
  $
  by
  $$
    F(f_1, \ldots, f_m)(x, j, i) = (f_i(x), j, i).
  $$
  It is easy to see that $F$ is a well-defined functor,
  and an isomorphism $\SetSurj^m \to \EE^+(\Theta)$.

  $(b)$ analogous to $(a)$.
\end{proof}

Let us first consider dual small Ramsey degrees.
Our intention is to apply Theorem~\ref{drdrs.thm.cofinal-amalg} to show that $\FinRelQuo^+(\Theta)$
has dual small Ramsey degrees. Let us show that $\FinE^+(\Theta)$ is a convenient subcategory of $\FinRelQuo^+(\Theta)$.

\begin{LEM}\label{drdrs.lem.small-degs-1}
  Let $\Theta = \{R_1, \ldots, R_m\}$ be a finite relational language.

  $(a)$ $\FinE(\Theta)$ is dually cofinal in $\FinRelQuo(\Theta)$.

  $(b)$ $\FinE^+(\Theta)$ is dually cofinal in $\FinRelQuo^+(\Theta)$.

  $(c)$ $\FinRelQuo(\Theta)$ has projective amalgamation.

  $(d)$ $\FinRelQuo^+(\Theta)$ has projective amalgamation.
\end{LEM}
\begin{proof}
  $(a)$ To see that $\FinE(\Theta)$ is dually cofinal in $\FinRelQuo(\Theta)$, take any finite $\Theta$-structure $\calA = (A, \Theta^A)$,
  let $Y$ be the set of isolated points in $A$ and let $n_i = |R_i^A|$, $i \le i \le m$. Then it is easy to construct a quotient map
  $\calE^Y_{n_1, \ldots, n_m} \to \calA$ (here $n_i = \{0, \ldots, n_i-1\}$).

  \medskip

  $(b)$ the same as $(a)$ with $Y = \0$.

  \medskip

  $(c)$
  Take arbitrary finite
  $\Theta$-structures $\calA = (A, \Theta^A)$, $\calB = (B, \Theta^B)$ and $\calC = (C, \Theta^C)$
  together with quotient maps $f : \calA \to \calC$ and $g : \calB \to \calC$.
  If $C = \0$ then $A = B = \0$ and $\calE^\0_{\0, \ldots, \0}$ together with the empty maps amalgamates the diagram.
  Assume, therefore, that $C \ne \0$. 
  
  For each pair $(a, b) \in A \times B$ such that $f(a) = g(b)$ add a new element $y_{a, b}$ to $Y$
  and put $p(y_{a, b}) = a$ and $q(y_{a, b}) = b$.
  
  As the next step fix an $i \in \{1, \ldots, m\}$ and let $h_i$ be the arity of $R_i$. If $R^\calC_i = \0$ put $X_i = \0$.
  Otherwise, for each pair of $h_i$-tuples $(\bar a, \bar b) \in R_i^A \times R_i^B$ such that
  $\bar a = (a_1, \ldots, a_{h_i})$, $\bar b = (b_1, \ldots, b_{h_i})$ and
  $(f(a_1), \ldots, f(a_{h_i})) = (g(b_1), \ldots, g(b_{h_i})) \in R_i^C$ add a new element $x_{\bar a, \bar b}$ to $X_i$
  which adds a new tuple $\big((x_{\bar a, \bar b}, 1, i), \ldots, (x_{\bar a, \bar b}, h_i, i)\big)$ to $R_i^{E^Y_{X_1, \ldots, X_m}}$,
  and put $p((x_{\bar a, \bar b}, j, i)) = a_j$ and $q((x_{\bar a, \bar b}, j, i)) = b_j$, $1 \le j \le h_i$.
  Then it is easy to check that $p : \calE^Y_{X_1, \ldots, X_M} \to \calA$ and
  $q : \calE^Y_{X_1, \ldots, X_M} \to \calB$ are quotient maps satisfying $f \circ p = g \circ q$.

  \medskip

  $(d)$ the same as $(c)$ with $Y = \0$.
\end{proof}

\begin{LEM}\label{drdrs.lem.dual-small-rd-no-iso-pts}
  Let $\Theta = \{R_1, \ldots, R_m\}$ be a finite relational language.
  
  $(a)$ The category $\FinE^+(\Theta)$ has dual small Ramsey degrees.

  $(b)$ The category $\FinRelQuo^+(\Theta)$ has dual small Ramsey degrees.
\end{LEM}
\begin{proof}
  $(a)$ The categories $\FinSetSurj^m$ and $\FinE^+(\Theta)$ are isomorphic by Lemma~\ref{drdrs.lem.E-iso-Set-m}, so
  the statement follows from Proposition~\ref{drdrs.prop.t-SetFin-n}.

  $(b)$ Follows immediately from the dual of Theorem~\ref{drdrs.thm.cofinal-amalg}, Lemma~\ref{drdrs.lem.small-degs-1} and~$(a)$.
\end{proof}

\begin{THM}
  Let $\Theta = \{R_1, \ldots, R_m\}$ be a finite relational language.
  The class of all finite $\Theta$-structures has dual small Ramsey degrees with respect to quotient maps.
\end{THM}
\begin{proof}
  The major technical issue with showing that $\FinRelQuo(\Theta)$ has dual small Ramsey degrees is the presence of
  isolated points. We shall bypass the issue by reducing the problem to $\FinRelQuo^+(\{R_0\} \cup \Theta)$, where
  $R_0 \notin \Theta$ is a new binary relational symbol.
  
  Just as a notational convenience, let
  $\Theta_0 = \{R_0\} \cup \Theta$, let $\CC = \FinRelQuo^+(\Theta_0)$ and $\BB = \FinRelQuo(\Theta)$. We know that $\CC$ has dual small Ramsey
  degrees (Lemma~\ref{drdrs.lem.dual-small-rd-no-iso-pts}), so in order to show that $\BB$ has dual small Ramsey degrees it
  suffices to establish a pre-adjunction $F : \Ob(\BB^\op) \rightleftarrows \Ob(\CC^\op) : G$.

  Take any $\calA = (A, \Theta^A) \in \Ob(\BB)$. Intuitively, to construct $F(\calA)$ we shall double $\calA$
  and then join each element with its double by a tuple in $R_0$. Formally, we let $F(\calA) = \calA' = (A', \Theta_0^{A'})$
  where:
  \begin{align*}
    A' &= \{0, 1\} \times A,\\
    R_0^{A'} &= \big\{((0,a), (1,a)) : a \in A\big\},
  \end{align*}
  and for $1 \le i \le m$:
  $$
  R_i^{A'} = \big\{\big((i,a_1), (i,a_2), \ldots, (i,a_{h_i})\big) : i \in \{0,1\},\; (a_1, \ldots, a_{h_i}) \in R_i^A \big\},
  $$
  where $h_i$ is the arity of $R_i$. Note that if $\calA$ had isolated points, the construction ensures that there are
  no isolated points in $F(\calA)$. Therefore, $F(\calA) \in \Ob(\CC)$.
  
  The construction of $G$ is much simpler. For a $\calC = (C, \Theta_0^C) \in \Ob(\CC)$ we will let $G(\calC)$ simply forget $R_0$.
  Formally,
  $$
    G\big((C, R_0^C, R_1^C, \ldots, R_m^C)\big) = (C, R_1^C, \ldots, R_m^C).
  $$
  Finally, let us construct the family
  $$
    \Phi_{\calA, \calC} : \hom_{\CC^\op}(F(\calA), \calC) \to \hom_{\BB^\op}(\calA, G(\calC)).
  $$
  Take any $\calA = (A, \Theta^A) \in \Ob(\BB)$, $\calC = (C, \Theta_0^C) \in \Ob(\CC)$ and a quotient map
  $u : \calC \to F(\calA)$. For a set $A$ define $\phi_A : \{0, 1\} \times A \to A$ by $\phi_A\big((i, a)\big) = a$,
  $i \in \{0,1\}$, $a \in A$, and then let $\Phi_{\calA, \calC}(u) = \phi_A \circ u$. It is easy to check that
  the definition of $\Phi$ is correct: $\Phi_{\calA, \calC}(u)$ is indeed a quotient map $G(\calC) \to \calA$.

  To complete the proof we still have to show that (PA) holds. Take any $\calC = (C, \Theta_0^C) \in \Ob(\CC)$,
  $\calB = (B, \Theta^B), \calD = (D, \Theta^D) \in \Ob(\BB)$ and let
  $f : \calD \to \calB$ be a quotient map. Define $f' : F(\calD) \to F(\calB)$ by
  $f'\big((i, d)\big) = (i, f(d))$, $i \in \{0,1\}$, $d \in D$. Clearly, $f'$ is a quotient map $F(\calD) \to F(\calB)$.
  Let $u : \calC \to F(\calD)$ be a quotient map and let us show that $f \circ \Phi_{\calD, \calC}(u) = \Phi_{\calB, \calC}(f' \circ u)$.
  Take any $c \in C$ and let $u(c) = (i, d)$. Then
  $$
    f \circ \Phi_{\calD, \calC}(u)(c) = f \circ \phi_D \circ u(c) = f \circ \phi_D (i, d) = f(d)
  $$
  while
  $$
  \Phi_{\calB, \calC}(f' \circ u)(c) = \phi_B \circ f' \circ u(c) = \phi_B \circ f'(i, d) = \phi_B(i, f(d)) = f(d).
  $$
  This concludes the proof that $F : \Ob(\BB^\op) \rightleftarrows \Ob(\CC^\op) : G$ is a pre-adjunction and the proof of the theorem.
\end{proof}

Before presenting the results pertaining to dual big Ramsey degrees we have to discuss a finesse.
Let $\Theta = \{R_1, \ldots, R_m\}$ be a finite relational language and let $\calA = (A, \Theta^A)$ and
$\calB = (B, \Theta^B)$ be $\Theta$-structures such that there exists a quotient map $\calA \to \calB$.
Then for every $i \in \{1, \ldots, m\}$, $R_i^A = \0$ iff $R_i^B = \0$. Consequently, no $\Theta$-structure
is projectively universal for $\FinRelQuo(\Theta)$, and the fact we have just discussed is the only reason.

If we picture $\RelQuo(\Theta)$ as an  oriented graph, it is easy to see that it splits into connected components
according to the choice of indices $i \in \{1, \ldots, m\}$ for which $R_i = \0$. More formally,
for a $\Theta$-structure $\calA = (A, \Theta^A)$ we say that the \emph{kind} of $\calA$ is a 01-sequence
$\bar b = (b_1, \ldots, b_m) \in \{0,1\}^m$ such that $b_i = 0$ iff $R_i^A = \0$.
For $\bar b = (b_1, \ldots, b_m) \in \{0,1\}^m$ let $\RelQuo(\Theta, \bar b)$ denote the full subcategory of
$\RelQuo(\Theta)$ spanned by all the $\Theta$-structures of kind $\bar b$. Then $\RelQuo(\Theta)$ is a disjoint union
of its subcategories $\RelQuo(\Theta, \bar b)$, $\bar b \in \{0,1\}^m$.

For every $\bar b \ne (0, 0, \ldots, 0)$ there is a $\Theta' \subseteq \Theta$ such that
$\RelQuo(\Theta, \bar b)$ is isomorphic to $\RelQuo\big(\Theta', (1, 1, \ldots, 1)\big)$: just select those
$R_i \in \Theta$ such that $b_i = 1$. This means that in order to understand the big Ramsey combinatorics of categories
$\RelQuo(\Theta)$ for various finite relational languages $\Theta$ it suffices to understand the big Ramsey combinatorics
of $\RelQuo^\full(\Theta) = \RelQuo\big(\Theta, (1, 1, \ldots, 1)\big)$.
Objects of $\RelQuo^\full(\Theta)$ will be referred to as \emph{full} $\Theta$-structures.

\begin{LEM}\label{drdrs.lem.big-drd-aux-1}
  Let $\Theta = \{R_1, \ldots, R_m\}$ be a finite relational language and
  let $X_1$, \ldots, $X_m$, $Y$ be finite sets such that $X_1$, \ldots, $X_m$ are nonempty. Then
  $$
    T^\partial_{\RelQuo(\Theta)}(\calE^Y_{X_1, \ldots, X_m}, \calE^\omega_{\omega, \ldots, \omega}) < \infty.
  $$
\end{LEM}
\begin{proof}
    Since $\EE(\Theta)$ is a full subcategory of $\RelQuo(\Theta)$ it suffices to show that
    $
      T^\partial_{\EE(\Theta)}(\calE^Y_{X_1, \ldots, X_m}, \calE^\omega_{\omega, \ldots, \omega}) < \infty
    $.

    For a quotient map $q : \calE^\omega_{\omega, \ldots, \omega} \to \calE^Y_{X_1, \ldots, X_m}$ the following set
    $$
      \tp(q) = \{q(x) : x \text{ is an isolated point of } \calE^\omega_{\omega, \ldots, \omega}\}
    $$
    will be referred as the \emph{type of $q$}. Clearly, $Y \times \{0\} \subseteq \tp(q) \subseteq E^Y_{X_1, \ldots, X_m}$.
    For a set $\tau$ satisfying $Y \times \{0\} \subseteq \tau \subseteq E^Y_{X_1, \ldots, X_m}$ let
    $$
      \QMap_\tau(\calE^\omega_{\omega, \ldots, \omega}, \calE^Y_{X_1, \ldots, X_m}) = 
        \{q \in \QMap(\calE^\omega_{\omega, \ldots, \omega}, \calE^Y_{X_1, \ldots, X_m}) : \tp(q) = \tau\}.
    $$

    \bigskip

    Claim. For every $\tau$ satisfying $Y \times \{0\} \subseteq \tau \subseteq E^Y_{X_1, \ldots, X_m}$ 
    there is a positive integer $t$ such that for every Borel coloring
    $
      \chi : \QMap_{\tau}(\calE^\omega_{\omega, \ldots, \omega}, \calE^Y_{X_1, \ldots, X_m}) \to k
    $
    there is a quotient map $w : \calE^\omega_{\omega, \ldots, \omega} \to \calE^\omega_{\omega, \ldots, \omega}$
    satisfying the following:
    \begin{itemize}
      \item $|\chi\big(\QMap_{\tau}(\calE^\omega_{\omega, \ldots, \omega}, \calE^Y_{X_1, \ldots, X_m}) \circ w\big)| \le t,$ and
      \item $w(x)$ is an isolated point if and only if $x$ is an isolated point. 
    \end{itemize}

    Proof. Let $t = T^\partial_{\SetSurj^{m+1}}\big((\tau, X_1, \ldots, X_m), (\omega, \omega, \ldots, \omega)\big)$, which is an
    integer by Proposition~\ref{drdrs.prop.T-Set-n}~$(c)$.
    
    Every quotient map $q \in \QMap_{\tau}(\calE^\omega_{\omega, \ldots, \omega}, \calE^Y_{X_1, \ldots, X_m})$ determines uniquely
    a tuple of surjective set-functions $(q_0, q_1, \ldots, q_m)$ where $q_0 : \omega \to \tau$ captures the behavior of $q$ on isolated points,
    while $q_i : \omega \to X_i$ captures the behavior of $q$ on the tuples related by $R_i$, $1 \le i \le m$. More precisely,
    $$
      q_0(s) = y \text{ if and only if } q\big((s, 0)\big) = y,
    $$
    while for $1 \le i \le m$:
    $$
      q_i(s) = x \text{ if and only if } q\big((s, j, i)\big) = (x, j, i),
    $$
    for all $j \in \{1, 2, \ldots, h_i\}$. We have thus established a \emph{bijective} map:
    $$
      \Phi : \QMap_{\tau}(\calE^\omega_{\omega, \ldots, \omega}, \calE^Y_{X_1, \ldots, X_m}) \to
      \Surj(\omega, \tau) \times \Surj(\omega, X_1) \times \ldots \times \Surj(\omega, X_m).
    $$
    It is easy to see that $\Phi$ is continuous.
    
    Let
    $
      \chi : \QMap_{\tau}(\calE^\omega_{\omega, \ldots, \omega}, \calE^Y_{X_1, \ldots, X_m}) \to k
    $
    be a finite Borel coloring. Define a finite coloring
    $$
      \gamma : \Surj(\omega, \tau) \times \Surj(\omega, X_1) \times \ldots \times \Surj(\omega, X_m) \to k
    $$
    by $\gamma(\Phi(f)) = \chi(f)$. This is a Borel coloring because $\chi$ is a Borel coloring and $\Phi$ is bijective and continuous.
    By Proposition~\ref{drdrs.prop.T-Set-n}~$(c)$ and the choice of $t$
    there exist surjective maps $w_0, w_1, \ldots, w_m : \omega \to \omega$ such that:
    $$
      \Big|\gamma\Big(\big(\Surj(\omega, \tau) \times \Surj(\omega, X_1) \times \ldots \times \Surj(\omega, X_m)\big)
      \circ (w_0, w_1, \ldots, w_m)\Big)\Big| \le t.
    $$
    Finally, define $w : \calE^\omega_{\omega, \ldots, \omega} \to \calE^\omega_{\omega, \ldots, \omega}$ as follows:
    $$
      w\big((s, 0)\big) = \big(w_0(s), 0\big)
    $$
    while for $1 \le i \le m$:
    $$
      w\big((s, j, i)\big) = \big(w_i(s), j, i\big),
    $$
    for all $j \in \{1, 2, \ldots, h_i\}$. It follows from the construction that $w$ is a quotient map such that
    $w(x)$ is an isolated point iff $x$ is an isolated point, while
    $$
      \big|\chi\big(\QMap_{\tau}(\calE^\omega_{\omega, \ldots, \omega}, \calE^Y_{X_1, \ldots, X_m}) \circ w\big)\big| \le t
    $$
    follows from the fact that
    $$
      \Phi(f \circ w) = \Phi(f) \circ (w_0, w_1, \ldots, w_m).
    $$
    This concludes the proof of the claim.

    \bigskip

    Going back to the proof of the lemma let $\tau_1, \ldots, \tau_s$ be an enumeration of all the sets $\tau$ satisfying
    $Y \times \{0\} \subseteq \tau \subseteq E^Y_{X_1, \ldots, X_m}$. Then, clearly,
    $$
      \QMap(\calE^\omega_{\omega, \ldots, \omega}, \calE^Y_{X_1, \ldots, X_m}) = \bigcup\nolimits_{i=1}^s
        \QMap_{\tau_i}(\calE^\omega_{\omega, \ldots, \omega}, \calE^Y_{X_1, \ldots, X_m}),
    $$
    and this is a disjoint union.
    Take any finite Borel coloring
    $$
      \chi : \QMap(\calE^\omega_{\omega, \ldots, \omega}, \calE^Y_{X_1, \ldots, X_m}) \to k.
    $$
    It is easy to see that all the restrictions
    $$
      \chi_i : \QMap_{\tau_i}(\calE^\omega_{\omega, \ldots, \omega}, \calE^Y_{X_1, \ldots, X_m}) \to k : f \mapsto \chi(f)
    $$
    of $\chi$ are also Borel colorings, $1 \le i \le s$. 
    Let us inductively construct a sequence $q_0, q_1, q_2, \ldots, q_s$ of
    quotient maps $\calE^\omega_{\omega, \ldots, \omega} \to \calE^\omega_{\omega, \ldots, \omega}$ as follows.
  
    To start the induction take $q_0$ to be the identity.
    For the induction step assume that $q_0, q_1, \ldots, q_{j-1}$ have been constructed
    and consider the Borel coloring $\gamma_j : \QMap_{\tau_j}(\calE^\omega_{\omega, \ldots, \omega}, \calE^Y_{X_1, \ldots, X_m}) \to k$ given by
    $$
      \gamma_j(f) = \chi_j(f \circ q_{j-1} \circ \ldots \circ q_1 \circ q_0).
    $$
    By the Claim, there is a positive integer $t_j$ and a
    quotient map $q_j : \calE^\omega_{\omega, \ldots, \omega} \to \calE^\omega_{\omega, \ldots, \omega}$ such that
    $$
      |\gamma_j(\QMap_{\tau_j}(\calE^\omega_{\omega, \ldots, \omega}, \calE^Y_{X_1, \ldots, X_m}) \circ q_j)| \le t_j.
    $$
    Note that, by construction, $q_j(x)$ is an isolated point iff $x$ is an isolated point. This property if $q_j$
    ensures that
    $$
      \QMap_{\tau_j}(\calE^\omega_{\omega, \ldots, \omega}, \calE^Y_{X_1, \ldots, X_m}) \circ q_j \subseteq
      \QMap_{\tau_j}(\calE^\omega_{\omega, \ldots, \omega}, \calE^Y_{X_1, \ldots, X_m}).
    $$

    Let $r = q_m \circ q_{m-1} \circ \ldots \circ q_1 \circ q_0$. Then $r$ is clearly a quotient map and
    $$
      \chi(\QMap(\calE^\omega_{\omega, \ldots, \omega}, \calE^Y_{X_1, \ldots, X_m}) \circ r) 
      = \bigcup\nolimits_{j=1}^s \chi_j(\QMap_{\tau_j}(\calE^\omega_{\omega, \ldots, \omega}, \calE^Y_{X_1, \ldots, X_m}) \circ r).
    $$
    By the construction of $q_j$'s it is now easy to see that for all $1 \le j \le s$:
    \begin{multline*}
      \QMap_{\tau_j}(\calE^\omega_{\omega, \ldots, \omega}, \calE^Y_{X_1, \ldots, X_m}) \circ r \subseteq\\
      \subseteq \QMap_{\tau_j}(\calE^\omega_{\omega, \ldots, \omega}, \calE^Y_{X_1, \ldots, X_m}) \circ q_j \circ q_{j-1} \circ \ldots \circ q_1 \circ q_0,
    \end{multline*}
    whence
    \begin{multline*}
      \bigcup\nolimits_{j=1}^s \chi_j(\QMap_{\tau_j}(\calE^\omega_{\omega, \ldots, \omega}, \calE^Y_{X_1, \ldots, X_m}) \circ r) \subseteq\\
      \subseteq \bigcup\nolimits_{j=1}^s \chi_j(\QMap_{\tau_j}(\calE^\omega_{\omega, \ldots, \omega}, \calE^Y_{X_1, \ldots, X_m}) \circ q_j \circ q_{j-1} \circ \ldots \circ q_1 \circ q_0) =\\
      = \bigcup\nolimits_{j=1}^s \gamma_j(\QMap_{\tau_j}(\calE^\omega_{\omega, \ldots, \omega}, \calE^Y_{X_1, \ldots, X_m}) \circ q_j).
    \end{multline*}
    Therefore,
    \begin{multline*}
      |\chi(\QMap(\calE^\omega_{\omega, \ldots, \omega}, \calE^Y_{X_1, \ldots, X_m}) \circ r)| \le\\
      \le \sum\nolimits_{j=1}^s |\gamma_j(\QMap_{\tau_j}(\calE^\omega_{\omega, \ldots, \omega}, \calE^Y_{X_1, \ldots, X_m}) \circ q_j)| \le\\
      \le \sum\nolimits_{j=1}^s t_j < \infty,
    \end{multline*}
    which concludes the proof.
\end{proof}

\begin{LEM}\label{drdrs.lem.big-drd-aux-2}
  Let $\Theta = \{R_1, \ldots, R_m\}$ be a finite relational language and let $\calA$ be a finite full
  $\Theta$-structure. Then there exist nonempty finite sets $X_1$, \ldots, $X_m$ and a possibly empty finite set $Y$
  such that there is a quotient map $\calE^Y_{X_1, \ldots, X_m} \to \calA$ and
  $\calE^\omega_{\omega, \ldots, \omega}$ is dually weakly homogeneous for $(\calA, \calE^Y_{X_1, \ldots, X_m})$
  in $\RelQuo(\Theta)$.
\end{LEM}
\begin{proof}
  Since $\FinE(\Theta)$ is dually cofinal in $\FinRelQuo(\Theta)$ (Lemma~\ref{drdrs.lem.small-degs-1})
  there exist finite sets $X_1$, \ldots, $X_m$, $Y$ such that there is a quotient map $\calE^Y_{X_1, \ldots, X_m} \to \calA$.
  Since $\calA$ is full, the existence of a quotient map $\calE^Y_{X_1, \ldots, X_m} \to \calA$ ensures that
  $X_1$, \ldots, $X_m$ are nonempty.

  To show that $\calE^\omega_{\omega, \ldots, \omega}$ is dually weakly homogeneous for $(\calA, \calE^Y_{X_1, \ldots, X_m})$
  in $\RelQuo(\Theta)$, fix an arbitrary quotient map $f : \calE^Y_{X_1, \ldots, X_m} \to \calA$ and take any
  quotient map $g : \calE^\omega_{\omega, \ldots, \omega} \to \calE^\omega_{\omega, \ldots, \omega}$ satisfying the following:
  \begin{itemize}
    \item for every $b \in E^\omega_{\omega, \ldots, \omega}$, isolated or not, there are infinitely many isolated points $a \in E^\omega_{\omega, \ldots, \omega}$
          with $g(a) = b$, and
    \item for every $i \in \{1, \ldots, m\}$ and every tuple $\bar b \in R_i^{E^\omega_{\omega, \ldots, \omega}}$
          there are infinitely many tuples $\bar a \in R_i^{E^\omega_{\omega, \ldots, \omega}}$ with $g(\bar a) = \bar b$
          (to keep the notation simple, for $\bar a = (a_1, \ldots, a_{h_i})$ we write $g(\bar a)$ to denote
          $(g(a_1), \ldots, g(a_{h_i}))$).
  \end{itemize}
  Let us show that
  $
    \QMap(\calE^\omega_{\omega, \ldots, \omega}, \calA) \circ g \subseteq f \circ \QMap(\calE^\omega_{\omega, \ldots, \omega}, \calE^Y_{X_1, \ldots, X_m})
  $.
  
  Take any $q \in \QMap(\calE^\omega_{\omega, \ldots, \omega}, \calA)$ and define
  $p \in \QMap(\calE^\omega_{\omega, \ldots, \omega}, \calE^Y_{X_1, \ldots, X_m})$ as follows.
  \begin{center}
    \begin{tikzcd}[column sep=large]
      \calE^Y_{X_1, \ldots, X_m} \arrow[<-,r,dashed,"p"] & \calE^\omega_{\omega, \ldots, \omega} \\
      \calA \arrow[<-,u, "f"] \arrow[<-,r, "q"'] & \calE^\omega_{\omega, \ldots, \omega} \arrow[<-,u, "g"']
    \end{tikzcd}
  \end{center}
  For an $i \in \{1, \ldots, m\}$ and a tuple $\bar a \in R_i^A$ let
  \begin{align*}
    S_{\bar a} &= \big\{\bar b \in R_i^{E^\omega_{\omega, \ldots, \omega}} : q(g(\bar b)) = \bar a \big\},\\
    T_{\bar a} &= \big\{\bar c \in R_i^{E^Y_{X_1, \ldots, X_m}} : f(\bar c) = \bar a \big\}.
  \end{align*}
  By the choice of $g$ we know that $S_{\bar a}$ is a countably infinite set, while
  $T_{\bar a}$ is a nonempty finite set. Define $p$ so that
  $$
    \{p(\bar b) : \bar b \in S_{\bar a} \} = T_{\bar a},
  $$
  and repeat this procedure for every $i \in \{1, \ldots, m\}$ and every $\bar a \in R_i^A$.

  Finally, let us handle isolated points. Take any $a \in A$, and let
  \begin{align*}
    S_a &= \{x \in E^\omega_{\omega, \ldots, \omega} : q(g(x)) = a \text{ and $x$ is isolated}\},\\
    T_a &= \{x \in E^Y_{X_1, \ldots, X_m} : f(x) = a \text{ and $x$ is isolated}\}.
  \end{align*}
  By the choice of $g$ we know that $S_{a}$ is a countably infinite set, while $T_a$ is finite.
  If $T_a \ne \0$, define $p$ so that $\{p(b) : b \in S_a\} = T_a$. If, however, $T_a = \0$
  take any $c \in E^Y_{X_1, \ldots, X_m}$ such that $f(c) = a$ and define $p$ so that
  $\{p(b) : b \in S_a\} = \{c\}$.

  The simple structure of objects in $\EE(\Theta)$ ensures that $p$ is a well-defined quotient map, while the construction of $p$
  ensures that $q \circ g = f \circ p$.
\end{proof}

\begin{LEM}\label{drdrs.lem.almost-final}
  Let $\Theta = \{R_1, \ldots, R_m\}$ be a finite relational language, and let $\calA$ be a nonempty finite $\Theta$-structure.
  If $\calA$ is of kind $(b_1, \ldots, b_m) \in \{0,1\}^m$ then
  $$
    T^\partial_{\RelQuo(\Theta)}(\calA, \calE^\omega_{b_1 \cdot \omega, \ldots, b_m \cdot \omega}) < \infty,
  $$
  where $1 \cdot \omega = \omega$ and $0 \cdot \omega = \0$.
\end{LEM}
\begin{proof}
  If $\calA$ is of the kind $(0, \ldots, 0)$ then $\calA$ is nothing but a set of isolated points, so
  $
    T^\partial_{\RelQuo(\Theta)}(\calA, \calE^\omega_{\0, \ldots, \0}) < \infty,
  $
  by the Infinite Dual Ramsey Theorem (Theorem~\ref{rpemklei.thm.IDRT}).

  If, on the other hand, $\calA$ is a full structure then
  $
    T^\partial_{\RelQuo(\Theta)}(\calA, \calE^\omega_{\omega, \ldots, \omega}) < \infty
  $
  follows from the dual of Theorem~\ref{rdbas.thm.mon-bigdeg} and
  Lemmas~\ref{drdrs.lem.big-drd-aux-1} and~\ref{drdrs.lem.big-drd-aux-2}.
  
  Finally, if $\calA$ is of kind $(b_1, \ldots, b_m) \ne (0,\ldots,0)$ then
  $$
    T^\partial_{\RelQuo(\Theta)}(\calA, \calE^\omega_{b_1 \cdot \omega, \ldots, b_m \cdot \omega}) < \infty
  $$
  follows from the above conclusion and the observation that
  for every kind $\bar b \ne (0, 0, \ldots, 0)$ there is a $\Theta' \subseteq \Theta$ such that
  $\RelQuo(\Theta, \bar b)$ is isomorphic to $\RelQuo^\full(\Theta')$.
\end{proof}

\begin{LEM}\label{drdrs.lem.bi-proj-2}
  Let $\Theta = \{R_1, \ldots, R_m\}$ be a finite relational language.
  Consider the category of all full $\Theta$ structures together with quotient maps,
  and let $\KK$ be the class of all finite full $\Theta$ structures
  such that $\calE^Y_{X_1, \ldots, X_m} \in \KK$ for all finite sets $X_1, \ldots, X_m, Y$ with
  $X_i \ne \0$, $1 \le i \le m$. Then a countable $\Theta$-structure $\calU$ is projectively universal for $\KK$
  if and only if $\calU$ is bi-projectible with $\calE^\omega_{\omega, \ldots, \omega}$.
\end{LEM}
\begin{proof}
  $(\Leftarrow)$ If $\calU$ is bi-projectible with $\calE^\omega_{\omega, \ldots, \omega}$
  then there is a quotient map $\calU \to \calE^\omega_{\omega, \ldots, \omega}$,
  and the statement now follows from the fact that $\calE^\omega_{\omega, \ldots, \omega}$ is
  projectively universal for~$\KK$.

  $(\Rightarrow)$ Assume that $\calU = (U, \Theta^U)$ is a countable $\Theta$-structure
  projectively universal for~$\KK$. The universality of $\calU$ implies that for every
  $n \in \NN$ there is a quotient map $q_n : \calU \to \calE^n_{n, \ldots, n}$.
  Then it is easy to see that there is a quotient map $\calU \to \calE^\omega_{\omega, \ldots, \omega}$.
  To complete the proof note that one can easily construct a quotient map $\calE^\omega_{\omega, \ldots, \omega} \to \calU$
  because $\calU$ is countable. Therefore, $\calU$ and $\calE^\omega_{\omega, \ldots, \omega}$ are bi-projectible.
\end{proof}

\begin{THM}\label{drdrs.thm.bigrd-general}
  Let $\Theta = \{R_1, \ldots, R_m\}$ be a finite relational language, let $\bar b \in \{0,1\}^m$ be arbitrary,
  let $\KK_{\bar b}$ be the class of all finite $\Theta$-structures of kind $\bar b$
  and let $\calU$ be a countable $\Theta$-structure projectively universal for $\KK_{\bar b}$. Then for every $\calA \in \KK_{\bar b}$:
  $$
    T^\partial_{\RelQuo(\Theta)}(\calA, \calU) < \infty.
  $$
\end{THM}
\begin{proof}
  This is an immediate consequence of Lemmas~\ref{drdrs.lem.bi-proj-general}, \ref{drdrs.lem.almost-final} and~\ref{drdrs.lem.bi-proj-2}
  using the fact that for every kind $\bar b \ne (0, 0, \ldots, 0)$ there is a $\Theta' \subseteq \Theta$ such that
  $\RelQuo(\Theta, \bar b)$ is isomorphic to $\RelQuo^\full(\Theta')$.
\end{proof}

\section{Reflexive binary structures}
\label{drdrs.sec.bin-rel-structs}

In this section we specialize the results of Section~\ref{drdrs.sec.rsfl} to some well-known classes of reflexive binary structures.
Let $\Theta = \{R\}$ be a relational language with a single binary letter.
A \emph{reflexive binary structure} is a $\Theta$-structure $\calA = (A, R^A)$ where $A$ is nonempty and
$R^A \subseteq A^2$ is a reflexive binary relation. Typical examples are \emph{reflexive graphs} (where $R^A$ is reflexive and
symmetric) and \emph{partially ordered sets} (where $R^A$ is reflexive, antisymmetric and transitive).
In addition to these two subclasses we shall also consider \emph{reflexive oriented graphs} (where $R^A$ is reflexive and antisymmetric),
\emph{reflexive transitive structures} (where $R^A$ is reflexive and transitive), and
\emph{reflexive acyclic structures} (where $R^A$ is reflexive and contains no cycles; a \emph{cycle} in $(A, R^A)$
is a sequence of points $x_1, \ldots, x_n \in A$, $n \ge 2$, such that $(x_1, x_2) \in R^A$, \ldots, $(x_{n-1}, x_n) \in R^A$
and $(x_n, x_1) \in R^A$). Note that a reflexive binary structure has no isolated points.

Let $\calA = (A, R^A)$ be a finite binary reflexive structure. Edges of the form $(a, a) \in R^A$ will be referred to
as \emph{loops}, while edges of the form $(a, b) \in R^A$ where $a \ne b$ will be referred to as \emph{non-loops}.
By $\widehat \calA = (\widehat A, R^{\widehat A})$ we denote the \emph{reflexive closure}
of $\calA$ where $\widehat A = A$ and $R^{\widehat A} = R^A \cup \{(a,a) : a \in A\}$.
By $\widetilde \calA = (\widetilde A, R^{\widetilde A})$ we denote the \emph{reflexive symmetric closure}
of $\calA$ where $\widetilde A = A$ and $R^{\widetilde A} = R^A \cup (R^A)^{-1} \cup \{(a,a) : a \in A\}$.

\begin{THM}\label{drdrs.thm.non-graphs}
  The following classes of finite structures have dual small Ramsey degrees with respect to quotient maps:

  $(a)$ the class $\KK_1$ of all finite reflexive binary structures;

  $(b)$ the class $\KK_2$ of finite reflexive oriented graphs;

  $(c)$ the class $\KK_3$ of finite reflexive acyclic structures;

  $(d)$ the class $\KK_4$ of finite reflexive transitive digraphs;

  $(e)$ the class $\KK_5$ of finite partial orders;

  $(f)$ the class $\KK_6$ of finite reflexive graphs.
\end{THM}
\begin{proof}
  $(a)$--$(e)$. Let $\Theta = \{R\}$ be a relational language with a single binary relational letter.
  By Lemma~\ref{drdrs.lem.dual-small-rd-no-iso-pts}~$(b)$ the category $\DD = \FinRelQuo^+(\Theta)$ has dual small Ramsey degrees.
  Moreover, by Lemma~\ref{drdrs.lem.small-degs-1}, for every $\calA$ there is a positive integer $t$ such that
  for every $\calB \in \Ob(\DD)$ and number of colors $k$ there is a finite set $X$ such that
  $\calE^\0_X \longrightarrow (\calB)^\calA_{k,t}$ in $\DD^\op$.
  
  Let $\CC$ be any of the categories listed in $(a)$--$(e)$ with quotient maps as morphisms.
  Clearly, $\CC$ is a subcategory of $\DD$, all the homsets in $\DD$ are finite and all the morphisms are epi.
  Let $G : \CC \to \DD$ be the embedding $G(C) = C$ and $G(f) = f$.
  In order to apply the dual of Proposition~\ref{crt.prop.D-solves-Delta-diag} we still have to
  show that $\CC$ is closed for finite dual amalgamation problems.

  Take any $\calA, \calB \in \Ob(\CC)$. Let $t = t^\partial_\DD(\calA)$ and let $k$ be an arbitrary positive integer.
  As we have seen above, there is a finite set $X$ such that $\calE^\0_X \longrightarrow (\calB)^\calA_{k,t}$ in $\DD^\op$.
  Let $F : \Delta^\op_G(\calA, \calB; \calE^\0_X) \to \CC$ be the diagram in $\CC$ constructed as in Proposition~\ref{crt.prop.D-solves-Delta-diag}
  so that $\calE^\0_X$ together with morphisms $e_i$, $i \in I$, solves the dual amalgamation problem.
  Let $\widehat{\calE^\0_X}$ be the reflexive closure of $\calE^\0_X$.
  It is easy to check that if $e_i : \calE^\0_X \to \calB$ is a quotient map, then $e_i$ is also a quotient map
  $\widehat{\calE^\0_X} \to \calB$ because $\calB$ is reflexive. Therefore, $\widehat{\calE^\0_X}$ is an object in $\CC$ which,
  together with morphisms $e_i$ $(i \in I)$, solves the same amalgamation problem in $\CC$.
  This shows that the dual of Proposition~\ref{crt.prop.D-solves-Delta-diag} applies whence
  $t^\partial_\CC(\calA) \le t = t^\partial_{\DD}(\calA)$.
  To conclude the proof note that $\widehat{\calE^\0_X}$ belongs to each of the classes $(a)$--$(e)$.

  \medskip

  $(f)$
  We know from $(a)$ that $\KK_1$ considered as a category of all finite reflexive binary structures
  and quotient maps has dual small Ramsey degrees. Clearly, $\KK_6$ is a subcategory of $\KK_1$, all the homsets in $\KK_1$ are
  finite and all the morphisms are epi. Therefore, in order to apply the dual of Proposition~\ref{crt.prop.D-solves-Delta-diag} we still have to
  show that $\KK_6$ is closed for finite dual amalgamation problems.

  Let us, first, note that if $\calA$ is a reflexive binary structure, $\calG$ is a reflexive graph and
  $f : \calA \to \calG$ is a quotient map, then the same $f$ is a quotient map $\widetilde \calA \to \calG$.
  Assume, now, that a dual finite amalgamation problem $F : \Delta \to \KK_6$ has a solution in $\KK_1$, and let
  $\calA \in \Ob(\KK_1)$ together with morphisms $f_1$, \ldots, $f_k$ be the compatible cone for the diagram $F$.
  Then $\widetilde \calA \in \Ob(\KK_6)$ together with the same morphisms $f_1$, \ldots, $f_k$ is
  a compatible cone in $\KK_6$ for $F$ showing, thus, that $\KK_6$ is closed for dual finite amalgamation problems.
  This shows that the dual of Proposition~\ref{crt.prop.D-solves-Delta-diag} applies whence
  $t^\partial_{\KK_6}(\calG) \le t^\partial_{\KK_1}(\calG) < \infty$ for all $\calG \in \Ob(\KK_6)$.
\end{proof}

Moving on to the dual big Ramsey degrees, let us first find a convenient projectively universal countable structure.

\begin{LEM}\label{drdrs.lem.bi-proj}
  $(a)$ Consider the category of all reflexive binary structures together with quotient maps and let $\KK_i$, $1 \le i \le 5$,
  be one of the classes listed in Theorem~\ref{drdrs.thm.non-graphs}.
  A countably infinite reflexive structure $\calU$ is projectively universal for $\KK_i$
  if and only if it is bi-projectible with $\widehat{\calE^\0_\omega}$
  if and only if it is bi-projectible with $\widehat{\calE^\omega_\omega}$.

  $(b)$ Consider the category of all reflexive graphs together with quotient maps and let $\KK_6$
  be the class of finite reflexive graphs as in Theorem~\ref{drdrs.thm.non-graphs}.
  A countably infinite reflexive graph $\calU$ is projectively universal for $\KK_6$
  if and only if it is bi-projectible with $\widetilde{\calE^\0_\omega}$
  if and only if it is bi-projectible with $\widetilde{\calE^\omega_\omega}$.
\end{LEM}
\begin{proof}
  Analogous to the proof of Lemma~\ref{drdrs.lem.bi-proj-2}. 
\end{proof}

\begin{LEM}\label{drdrs.lem.big-drd-aux-22}
  $(a)$ $T^\partial(\calA, \widehat{\calE^\omega_\omega}) < \infty$ for every finite reflexive binary structure~$\calA$.

  $(b)$ $T^\partial(\calG, \widetilde{\calE^\omega_\omega}) < \infty$ for every finite reflexive graph~$\calG$.
\end{LEM}
\begin{proof}
  As a notational convenience let $\CC^* = \RelQuo(\Theta)$ where $\Theta = \{R\}$ is a relational language with
  a single binary letter. 

  $(a)$ Let $\CC$ be the category of all reflexive binary relational structures and quotient maps, and
  let $U : \CC^* \to \CC$ be the functor defined by $U(\calA) = \widehat \calA$ on objects
  and $U(f) = f$ on morphisms. Note that $U$ is well-defined on morphisms: if $f : \calA \to \calB$ is a quotient map
  then the same $f$ is also a quotient map $\widehat\calA \to \widehat\calB$ between their reflexive closures.
  Let us show that $U : \CC^* \to \CC$ has dual restrictions. Take any $\calA = (A, R^A) \in \Ob(\CC^*)$,
  $\calB = (B, R^B) \in \Ob(\CC)$ and a quotient map $q : \widehat \calA \to \calB$. Let
  $\calB' = (B', R^{B'})$ where $B' = B$ and
  $
    R^{B'} = \{(q(x), q(y)) : (x, y) \in R^A \}
  $.
  Then $q : \calA \to \calB'$ is a quotient map and $\widehat{\calB'} = \calB$.
  \begin{center}
    \begin{tikzcd}
        \calB' \ar[d, dashed, mapsto, "U"'] & \calA \ar[l, "q"'] \ar[d, dashed, mapsto, "U"]\\
        \llap{$\widehat{\calB'} = $\;}\calB &  \widehat\calA \ar[l, "q"]
    \end{tikzcd}
  \end{center}
  Moreover, for each $\calX, \calY \in \Ob(\CC^*)$ we have that
  $
    U_{\calX\calY} : \hom_{\CC^*}(\calX, \calY) \to \hom_\CC(\widehat\calX, \widehat\calY)
  $
  is a continuous and injective map between the two Polish spaces. Therefore, each $U_{\calX\calY}$ takes a Borel set onto a Borel set.
  So, the assumptions of the dual of Theorem~\ref{sbrd.thm.big1-dual-borel-weaker} are satisfied and we get:
  $$
    T^\partial_\CC(\calA, \widehat{\calE^\omega_\omega}) \le \sum_{\calB \in U^{-1}(\calA)} T^\partial_{\CC^*}(\calB, \calE^\omega_\omega).
  $$
  The statement now follows from the fact that $U^{-1}(\calA)$ is finite and
  $T^\partial_{\CC^*}(\calB, \calE^\omega_\omega) < \infty$ (Lemma~\ref{drdrs.lem.almost-final}).

  $(b)$ Let $\CC$ be the category of all reflexive and symmetric binary relational structures and quotient maps, and
  let $U : \CC^* \to \CC$ be the functor defined by $U(\calA) = \widetilde \calA$ on objects
  and $U(f) = f$ on morphisms. By repeating the argument from $(a)$ we show that
  $U$ is well-defined on morphisms and that the assumptions of the dual of Theorem~\ref{sbrd.thm.big1-dual-borel-weaker} are satisfied.
  Therefore,
  $$
    T^\partial_\CC(\calA, \widetilde{\calE^\omega_\omega}) \le \sum_{\calB \in U^{-1}(\calA)} T^\partial_{\CC^*}(\calB, \calE^\omega_\omega) < \infty.
  $$
  This concludes the proof.
\end{proof}

\begin{THM}\label{drdrs.thm.big-ram-deg-nongraphs}
  Consider one of the following classes of structures:
  \begin{itemize}\itemsep -1pt
    \item all finite and countably infinite reflexive binary structures,
    \item all finite and countably infinite reflexive oriented graphs,
    \item all finite and countably infinite acyclic reflexive digraphs,
    \item all finite and countably infinite transitive reflexive digraphs,
    \item all finite and countably infinite partial orders,
    \item all finite and countably infinite reflexive graphs,
  \end{itemize}
  with respect to quotient maps. Let $\calU$ be a countably infinite structure in the class which is
  projectively universal for all the finite structures in the class. Then for every finite $\calA$ in the class we have that
  $T^\partial(\calA, \calU) < \infty$.
\end{THM}
\begin{proof}
  Follows immediately from Lemmas~\ref{drdrs.lem.bi-proj-general}, \ref{drdrs.lem.bi-proj} and~\ref{drdrs.lem.big-drd-aux-22}.
\end{proof}

In sharp contrast to the classes of binary reflexive structures we have considered above,
we are now going to show that the category $\TT^\fin$ of finite reflexive tournaments and surjective homomorphisms
does not have dual small Ramsey degrees. We will also show that no reflexive tournament is projectively universal
for the class of all finite reflexive tournaments, so no structure in $\Ob(\TT)$ has big dual Ramsey degrees.

A \emph{reflexive tournament} is a relational structure $(T, \Boxed\to)$ where
$\Boxed\to \subseteq T^2$ is a reflexive binary relation such that
for all $x \ne y$ exactly one of $x \to y$ or $y \to x$ holds. 
For any positive integer $n$ the \emph{reflexive acyclic tournament $A_n$} is a reflexive tournament
whose vertices are $\{0, 1, \ldots, n-1\}$ and where $x \to y$ if and only if $x \le y$.
If $(T, \Boxed\to)$ is a reflexive tournament and $X \subseteq T$, then $T[X]$ denotes the \emph{subtournament of $T$ induced by $X$},
that is, a tournament $(X, \Boxed\to)$ such that $x_1 \to x_2$ in $X$ if and only if $x_1 \to x_2$ in $T$.

If $S$ and $T$ are tournaments then $S \Rightarrow T$ denotes a tournament on the disjoint union $S \mathbin{\dot\cup} T$ where
the arrows within $S$ and $T$ are preserved and $s \to t$ for every $s \in S$ and every $t \in T$.
Note that $\Rightarrow$ is associative in the sense that $S \Rightarrow (T \Rightarrow U) \cong (S \Rightarrow T) \Rightarrow U$,
so it is safe to omit parentheses in longer expressions of this form.

A mapping $f : S \to T$ is a \emph{homomorphism} of reflexive tournaments
$(S, \Boxed\to)$ and $(T, \Boxed\to)$ if $x \to y$ in $S$ implies $f(x) \to f(y)$ in $T$, for all $x, y \in S$.
Note that every surjective homomorphism $f : (S, \Boxed\to) \to (T, \Boxed\to)$
is actually a quotient map.

We say that a reflexive tournament $(T, \Boxed\to)$ is an \emph{inflation} of a reflexive tournament $(S, \Boxed\to)$
if there is a surjective homomorphism from $(T, \Boxed\to)$ onto $(S, \Boxed\to)$. We say that reflexive tournaments $(S_1, \Boxed\to)$ and $(S_2, \Boxed\to)$ are \emph{siblings}
if there exists a reflexive tournament $(T, \Boxed\to)$ which is an inflation of both $(S_1, \Boxed\to)$ and
$(S_2, \Boxed\to)$.

\begin{LEM}\label{drdrs.lem.nonsiblings}
    There exists a countable family $\{T_n : n \ge 3 \text{ is odd}\}$ of finite reflexive tournaments 
    no two of which are siblings.
\end{LEM}
\begin{proof}
    For each odd $n \ge 3$ define a tournament $T_n$ as follows: label the vertices by $1, 2, \ldots, n$
    in a cyclic order and then orient the edges in such a way that for each vertex $v$,
    each of the previous $\frac{n-1}{2}$ vertices dominate $v$, while $v$ dominates the remaining $\frac{n-1}{2}$ vertices.
    Let us show that $T_m$ and $T_n$ are not siblings whenever $m$ and $n$ are distinct odd integers.

    Suppose this is not the case. Then there exist integers $m < n$ with $m = 2p+1$ and $n = 2q + 1$,
    and a tournament $T$ with surjective homomorphisms $f : T \to T_m$ and $g : T \to T_n$.
    Define a 01-matrix $A = [a_{i,j}]_{m\times n}$ so that $a_{i,j}$ is 1 if and only if there exists
    a vertex $v$ in $T$ such that $f(v) = i$ and $g(v) = j$; otherwise we let $a_{i,j} = 0$.

    Clearly, in each row and and in each column at least one entry has to be 1.
    Since there are at least $n > m$ ones in the matrix, there is a row in which at least two entries are~1.
    Note that we can cyclically permute
    columns and rows of the matrix and still obtain a matrix that corresponds to a pair of surjective homomorphisms
    $T \to T_m$ and $T \to T_n$ (because cyclic permutations are automorphisms of every $T_k$ in the family).
    Therefore, we can cyclically permute the rows of the matrix until we reach the position where the last row contains
    at least two 1's, and then we can cyclically permute the columns of the matrix until we reach the position where
    $a_{m,1} = 1$ and $a_{m, s} = 1$ for some $s \in \{2, 3, \ldots, q+1\}$.
    Without loss of generality we may assume that
    this matrix corresponds to the surjective homomorphisms $f : T \to T_m$ and $g : T \to T_n$.
    Let $u$ be a vertex in $T$ such that $f(u) = m$ and $g(u) = 1$, and let
    $v$ be a vertex in $T$ such that $f(v) = m$ and $g(v) = s$.

    \medskip

    Claim 1. $a_{2p,j} = 0$ for all $j \in \{2, 3, \ldots, q+1\}$.
    
    Proof. Assume that $a_{2p,j} = 1$ for some $j \in \{2, 3, \ldots, q+1\}$. Then there is a vertex $x$ in $T$ such that $f(x) = 2p$ and $g(x) = j$.
    Because $2p \to m$ in $T_m$ it follows that $x \to u$ in $T$. On the other hand, because $1 \to j$ in $T_n$
    it follows that $u \to x$ in $T$. Contradiction.

    \medskip

    Claim 2. $a_{i,q+2} = 0$ for all $i \in \{1, 2, \ldots, p\}$.

    Proof. Analogous to the proof of Claim 1: if $a_{i,q+2} = 1$ for some $i \in \{1, 2, \ldots, p\}$
    then we get the contradiction from the fact that $m \to i$ in $T_m$ and $q+2 \to 1$ in $T_n$.

    \medskip

    Claim 3. $a_{i,q+2} = 0$ for all $i \in \{p+1, p+2, \ldots, 2p\}$.

    Proof. The proof of this claim uses the fact that $a_{m, s} = 1$.
    Assume that $a_{i, q+2} = 1$ for some $i \in \{p+1, p+2, \ldots, 2p\}$.
    Then there is a vertex $z$ in $T$ such that $f(z) = i$ and $g(z) = q+2$.
    Because $i \to m$ in $T_m$ it follows that $z \to v$ in $T$. On the other hand, because $s \to q+2$ in $T_n$
    it follows that $v \to z$ in $T$. Contradiction.

    \medskip

    Claim 4. $a_{m,q+2} = 1$.

    Proof. By Claims 2 and 3 all the other entries in column $q+2$ are zero, so $a_{m,q+2}$ must be 1 because every column
    contains at least one entry equal to~1.

    \medskip

    Claim 5. $a_{2p,j} = 0$ for all $j \in \{q+3, \ldots, n, 1\}$.

    Proof. Analogous to the proof of Claim 1 using Claim 4: if $a_{2p,j} = 1$ for some $j \in \{q+3, \ldots, n, 1\}$
    then we get the contradiction from the fact that $2p \to m$ in $T_m$ and $q+2 \to j$ in $T_n$.

    \medskip

    Claims 1, 3 and 5 now yield that all the entries in row $2p$ are 0 -- contradiction with the fact that every column
    contains at least one entry equal to~1.
\end{proof}

\begin{THM}\label{drdrs.thm.fintrounaments}
  The class of finite reflexive tournaments taken with surjective homomorphisms
  does not have dual small Ramsey degrees.
\end{THM}
\begin{proof}
  Let us show that with respect to surjective homomorphisms the acyclic three-element tournament $A_3$
  does not have a dual small Ramsey degree in the class of all finite reflexive tournaments.
  Let $T_1$, $T_2$, $T_3$, \ldots, be a family of finite reflexive tournaments no two of which are siblings, say a family
  given in Lemma~\ref{drdrs.lem.nonsiblings}. Take any $n \in \NN$ and let $B_n$ denote the following tournament:
  $$
    1 \Rightarrow T_1 \Rightarrow T_2 \Rightarrow \ldots \Rightarrow T_n \Rightarrow 1,
  $$
  where $1$ denotes the trivial one-element reflexive tournament.
  Take any tournament $S$ with $\Surj(S, B_n) \ne \0$ (here, $\Surj(S, T)$ denotes the set of all surjective homomorphisms
  from $S$ to $T$), and consider the coloring
  $$
    \chi : \Surj(S, A_3) \to \{1, 2, \ldots, n\}
  $$
  defined as follows: if there is a $j \in \{1, 2, \ldots, n\}$ such that 
  $S[f^{-1}(1)]$ is an inflation of $T_j$ we let $\chi(f) = j$; if, however, $S[f^{-1}(1)]$ is an inflation of $T_j$ for no
  $j \in \{1, 2, \ldots, n\}$ we let $\chi(f) = 1$.
  By the choice of the family $T_1$, $T_2$, $T_3$, \ldots, it is not possible for $S[f^{-1}(1)]$ to be both
  an inflation of $T_i$ and an inflation of $T_j$ for some $i \ne j$, so $\chi$ is well defined.

  \begin{figure}
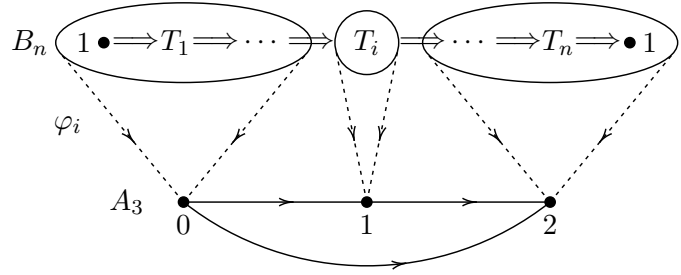

    \centering
    {\def\R{$\Longrightarrow$}%
\begin{pgfpicture}
  \pgfsetxvec{\pgfpoint{\acadpgfunit}{0pt}}
  \pgfsetyvec{\pgfpoint{0pt}{\acadpgfunit}}
  \pgfsetlinewidth{\acadpgflinewidth}
  \pgftransformshift{\pgfpointxy{0.0}{0.0}}

  \begin{pgfscope}
    \pgfpathellipse{\pgfpointxy{312.5}{375.0}}{\pgfpointxy{200.39}{0.0}}{\pgfpointxy{0.0}{62.5}}
    \pgfusepath{stroke}
  \end{pgfscope}
  \begin{pgfscope}
    \pgfpathellipse{\pgfpointxy{887.5}{375.0}}{\pgfpointxy{200.0}{0.0}}{\pgfpointxy{0.0}{62.5}}
    \pgfusepath{stroke}
  \end{pgfscope}
  \begin{pgfscope}
    \pgfpathellipse{\pgfpointxy{600.0}{375.0}}{\pgfpointxy{50.0}{0.0}}{\pgfpointxy{0.0}{50.0}}
    \pgfusepath{stroke}
  \end{pgfscope}
  \begin{pgfscope}
    \pgfpathmoveto{\pgfpointxy{312.5}{125.0}}
    \pgfpathlineto{\pgfpointxy{600.0}{125.0}}
    \pgfusepath{stroke}
  \end{pgfscope}
  \begin{pgfscope}
    \pgfpathmoveto{\pgfpointxy{600.0}{125.0}}
    \pgfpathlineto{\pgfpointxy{887.5}{125.0}}
    \pgfusepath{stroke}
  \end{pgfscope}
  \begin{pgfscope}
    \pgfpathmoveto{\pgfpointxy{462.5}{131.029}}
    \pgfpatharcaxes{240.0}{270.0}{\pgfpointxy{45.0}{0.0}}{\pgfpointxy{0.0}{45.0}}
    \pgfusepath{stroke}
  \end{pgfscope}
  \begin{pgfscope}
    \pgfpathmoveto{\pgfpointxy{485.0}{125.0}}
    \pgfpatharcaxes{90.0}{120.0}{\pgfpointxy{45.0}{0.0}}{\pgfpointxy{0.0}{45.0}}
    \pgfusepath{stroke}
  \end{pgfscope}
  \begin{pgfscope}
    \pgfpathmoveto{\pgfpointxy{750.0}{131.029}}
    \pgfpatharcaxes{240.0}{270.0}{\pgfpointxy{45.0}{0.0}}{\pgfpointxy{0.0}{45.0}}
    \pgfusepath{stroke}
  \end{pgfscope}
  \begin{pgfscope}
    \pgfpathmoveto{\pgfpointxy{772.5}{125.0}}
    \pgfpatharcaxes{90.0}{120.0}{\pgfpointxy{45.0}{0.0}}{\pgfpointxy{0.0}{45.0}}
    \pgfusepath{stroke}
  \end{pgfscope}
  \begin{pgfscope}
    \pgfpathmoveto{\pgfpointxy{312.5}{125.0}}
    \pgfpatharcaxes{231.642}{308.358}{\pgfpointxy{463.281}{0.0}}{\pgfpointxy{0.0}{463.281}}
    \pgfusepath{stroke}
  \end{pgfscope}
  \begin{pgfscope}
    \pgfpathmoveto{\pgfpointxy{638.833}{32.5095}}
    \pgfpatharcaxes{245.701}{277.672}{\pgfpointxy{42.226}{0.0}}{\pgfpointxy{0.0}{42.226}}
    \pgfusepath{stroke}
  \end{pgfscope}
  \begin{pgfscope}
    \pgfpathmoveto{\pgfpointxy{661.846}{29.1466}}
    \pgfpatharcaxes{97.6716}{129.642}{\pgfpointxy{42.226}{0.0}}{\pgfpointxy{0.0}{42.226}}
    \pgfusepath{stroke}
  \end{pgfscope}
  \begin{pgfscope}
    \pgfsetdash{{1.5pt}{2pt}}{0pt}
    \pgfpathmoveto{\pgfpointxy{312.5}{125.0}}
    \pgfpathlineto{\pgfpointxy{506.527}{359.375}}
    \pgfusepath{stroke}
  \end{pgfscope}
  \begin{pgfscope}
    \pgfsetdash{{1.5pt}{2pt}}{0pt}
    \pgfpathmoveto{\pgfpointxy{312.5}{125.0}}
    \pgfpathlineto{\pgfpointxy{118.473}{359.375}}
    \pgfusepath{stroke}
  \end{pgfscope}
  \begin{pgfscope}
    \pgfsetdash{{1.5pt}{2pt}}{0pt}
    \pgfpathmoveto{\pgfpointxy{887.5}{125.0}}
    \pgfpathlineto{\pgfpointxy{1081.53}{359.375}}
    \pgfusepath{stroke}
  \end{pgfscope}
  \begin{pgfscope}
    \pgfsetdash{{1.5pt}{2pt}}{0pt}
    \pgfpathmoveto{\pgfpointxy{887.5}{125.0}}
    \pgfpathlineto{\pgfpointxy{693.473}{359.375}}
    \pgfusepath{stroke}
  \end{pgfscope}
  \begin{pgfscope}
    \pgfsetdash{{1.5pt}{2pt}}{0pt}
    \pgfpathmoveto{\pgfpointxy{600.0}{125.0}}
    \pgfpathlineto{\pgfpointxy{648.99}{365.0}}
    \pgfusepath{stroke}
  \end{pgfscope}
  \begin{pgfscope}
    \pgfsetdash{{1.5pt}{2pt}}{0pt}
    \pgfpathmoveto{\pgfpointxy{600.0}{125.0}}
    \pgfpathlineto{\pgfpointxy{551.01}{365.0}}
    \pgfusepath{stroke}
  \end{pgfscope}
  \begin{pgfscope}
    \pgfpathmoveto{\pgfpointxy{581.811}{244.251}}
    \pgfpatharcaxes{161.537}{191.537}{\pgfpointxy{45.0}{0.0}}{\pgfpointxy{0.0}{45.0}}
    \pgfusepath{stroke}
  \end{pgfscope}
  \begin{pgfscope}
    \pgfpathmoveto{\pgfpointxy{580.404}{221.0}}
    \pgfpatharcaxes{11.537}{41.537}{\pgfpointxy{45.0}{0.0}}{\pgfpointxy{0.0}{45.0}}
    \pgfusepath{stroke}
  \end{pgfscope}
  \begin{pgfscope}
    \pgfpathmoveto{\pgfpointxy{630.003}{241.84}}
    \pgfpatharcaxes{138.463}{168.463}{\pgfpointxy{45.0}{0.0}}{\pgfpointxy{0.0}{45.0}}
    \pgfusepath{stroke}
  \end{pgfscope}
  \begin{pgfscope}
    \pgfpathmoveto{\pgfpointxy{619.596}{221.0}}
    \pgfpatharcaxes{-11.537}{18.463}{\pgfpointxy{45.0}{0.0}}{\pgfpointxy{0.0}{45.0}}
    \pgfusepath{stroke}
  \end{pgfscope}
  \begin{pgfscope}
    \pgfpathmoveto{\pgfpointxy{800.185}{239.926}}
    \pgfpatharcaxes{189.62}{219.62}{\pgfpointxy{45.0}{0.0}}{\pgfpointxy{0.0}{45.0}}
    \pgfusepath{stroke}
  \end{pgfscope}
  \begin{pgfscope}
    \pgfpathmoveto{\pgfpointxy{809.889}{218.75}}
    \pgfpatharcaxes{39.6196}{69.6196}{\pgfpointxy{45.0}{0.0}}{\pgfpointxy{0.0}{45.0}}
    \pgfusepath{stroke}
  \end{pgfscope}
  \begin{pgfscope}
    \pgfpathmoveto{\pgfpointxy{984.103}{232.237}}
    \pgfpatharcaxes{110.38}{140.38}{\pgfpointxy{45.0}{0.0}}{\pgfpointxy{0.0}{45.0}}
    \pgfusepath{stroke}
  \end{pgfscope}
  \begin{pgfscope}
    \pgfpathmoveto{\pgfpointxy{965.111}{218.75}}
    \pgfpatharcaxes{320.38}{350.38}{\pgfpointxy{45.0}{0.0}}{\pgfpointxy{0.0}{45.0}}
    \pgfusepath{stroke}
  \end{pgfscope}
  \begin{pgfscope}
    \pgfpathmoveto{\pgfpointxy{225.185}{239.926}}
    \pgfpatharcaxes{189.62}{219.62}{\pgfpointxy{45.0}{0.0}}{\pgfpointxy{0.0}{45.0}}
    \pgfusepath{stroke}
  \end{pgfscope}
  \begin{pgfscope}
    \pgfpathmoveto{\pgfpointxy{234.889}{218.75}}
    \pgfpatharcaxes{39.6196}{69.6196}{\pgfpointxy{45.0}{0.0}}{\pgfpointxy{0.0}{45.0}}
    \pgfusepath{stroke}
  \end{pgfscope}
  \begin{pgfscope}
    \pgfpathmoveto{\pgfpointxy{409.103}{232.237}}
    \pgfpatharcaxes{110.38}{140.38}{\pgfpointxy{45.0}{0.0}}{\pgfpointxy{0.0}{45.0}}
    \pgfusepath{stroke}
  \end{pgfscope}
  \begin{pgfscope}
    \pgfpathmoveto{\pgfpointxy{390.111}{218.75}}
    \pgfpatharcaxes{320.38}{350.38}{\pgfpointxy{45.0}{0.0}}{\pgfpointxy{0.0}{45.0}}
    \pgfusepath{stroke}
  \end{pgfscope}
  \begin{pgfscope}
    \pgfsetfillcolor{black}
    \pgfpathellipse{\pgfpointxy{1012.5}{375.0}}{\pgfpointxy{8.0}{0.0}}{\pgfpointxy{0.0}{8.0}}
    \pgfusepath{fill,stroke}
  \end{pgfscope}
  \begin{pgfscope}
    \pgfsetfillcolor{black}
    \pgfpathellipse{\pgfpointxy{187.5}{375.0}}{\pgfpointxy{8.0}{0.0}}{\pgfpointxy{0.0}{8.0}}
    \pgfusepath{fill,stroke}
  \end{pgfscope}
  \begin{pgfscope}
    \pgfsetfillcolor{black}
    \pgfpathellipse{\pgfpointxy{600.0}{125.0}}{\pgfpointxy{8.0}{0.0}}{\pgfpointxy{0.0}{8.0}}
    \pgfusepath{fill,stroke}
  \end{pgfscope}
  \begin{pgfscope}
    \pgfsetfillcolor{black}
    \pgfpathellipse{\pgfpointxy{887.5}{125.0}}{\pgfpointxy{8.0}{0.0}}{\pgfpointxy{0.0}{8.0}}
    \pgfusepath{fill,stroke}
  \end{pgfscope}
  \begin{pgfscope}
    \pgfsetfillcolor{black}
    \pgfpathellipse{\pgfpointxy{312.5}{125.0}}{\pgfpointxy{8.0}{0.0}}{\pgfpointxy{0.0}{8.0}}
    \pgfusepath{fill,stroke}
  \end{pgfscope}
  \pgftext[at={\pgfpointxy{600.0}{375.0}}]{$T_i$}
  \pgftext[at={\pgfpointxy{762.5}{375.0}}]{$\cdots$}
  \pgftext[at={\pgfpointxy{900.0}{375.0}}]{$T_n$}
  \pgftext[at={\pgfpointxy{300.0}{375.0}}]{$T_1$}
  \pgftext[at={\pgfpointxy{437.5}{375.0}}]{$\cdots$}
  \pgftext[at={\pgfpointxy{237.5}{375.0}}]{\R}
  \pgftext[at={\pgfpointxy{362.5}{375.0}}]{\R}
  \pgftext[at={\pgfpointxy{512.5}{375.0}}]{\R}
  \pgftext[at={\pgfpointxy{687.5}{375.0}}]{\R}
  \pgftext[at={\pgfpointxy{837.5}{375.0}}]{\R}
  \pgftext[at={\pgfpointxy{962.5}{375.0}}]{\R}
  \pgftext[left,at={\pgfpointxy{1032.5}{375.0}}]{1}
  \pgftext[right,at={\pgfpointxy{167.5}{375.0}}]{1}
  \pgftext[right,at={\pgfpointxy{100.5}{375.0}}]{$B_n$}
  \pgftext[top,at={\pgfpointxy{312.5}{105.0}}]{0}
  \pgftext[top,at={\pgfpointxy{600.0}{105.0}}]{1}
  \pgftext[top,at={\pgfpointxy{887.5}{105.0}}]{2}
  \pgftext[right,at={\pgfpointxy{250.5}{125.0}}]{$A_3$}
  \pgftext[at={\pgfpointxy{130.912}{244.994}}]{$\phi_i$}
\end{pgfpicture}
    }
    \caption{The construction from the proof of Theorem~\ref{drdrs.thm.fintrounaments}}
    \label{drdrs.fig.phi}
  \end{figure}
  
  Take any $w \in \Surj(S, B_n)$ and consider $\phi_1, \phi_2, \ldots, \phi_n \in \Surj(B_n, A_3)$
  where $\phi_i$, $1 \le i \le n$,
  is defined so that $\phi_i^{-1}(1) = T_i$, $\phi_i^{-1}(0)$ contains all the vertices of $B_n$
  that arrow $T_i$ and $\phi_i^{-1}(2)$ contains all the vertices of $B_n$ that $T_i$ arrows,
  see Fig.~\ref{drdrs.fig.phi}. It is now clear that for $1 \le i \le n$ we have that
  $S[(\phi_i \circ w)^{-1}(1)]$ is an inflation of $T_i$ whence
  $\chi(\phi_i \circ w) = i$. Therefore, $|\chi(\Surj(B_n, A_3) \circ w)| = n$.
  This completes the proof.
\end{proof}

It is even easier to settle the case of dual big Ramsey degrees for tournaments.
Lemma~\ref{drdrs.lem.nonsiblings} exhibits an infinite family of finite reflexive tournaments no two of which are siblings.
Therefore, a tournament which is projectively universal for the class of all finite reflexive tournaments cannot exist.
Consequently, there is no countably infinite reflexive tournament $\calS$ with 
$T^\partial(\calT, \calS) < \infty$ for all finite reflexive tournaments $\calT$.

\section{Metric spaces}
\label{drdrs.sec.met-spc}

Finally, let us take a look at metric spaces which can be understood as special relational structures over infinite binary languages.
A mapping $f : (M, d) \to (M', d')$ between two metric spaces is \emph{non-expansive} if $d'(f(x), f(y)) \le d(x, y)$ for all $x, y \in M$.
For a metric space $\calM = (M, d)$ let
$\Spec(\calM) = \{d(x, y) : x, y \in M \text{ and } x \ne y\}$. For a class $\CC$ of metric spaces let
$$
  \Spec(\CC) = \bigcup \{\Spec(\calM) : \calM \in \Ob(\CC)\}.
$$

A \emph{linearly ordered metric space} is a triple $\calM = (M, d, \Boxed\sqsubset)$
where $(M, d)$ is a metric space and $\sqsubset$ is a linear order on~$M$.
For a positive integer $n$ and a positive real number $\delta$ let
$$
\calE_{n,\delta} = (\{1, 2, \ldots, n\}, d_n^\delta) \text{\quad and\quad}
  \calE^<_{n,\delta} = (\{1, 2, \ldots, n\}, d_n^\delta, \Boxed<)
$$
where $<$ is the usual ordering of the integers and $d_n^\delta(x, y) = \delta$ whenever $x \ne y$.
A mapping $f : M \to M'$ is a
\emph{non-expansive rigid surjection} from $(M, d, \Boxed\sqsubset)$ to $(M', d', \Boxed{\sqsubset'})$ if
$f : (M, d) \to (M', d')$ is non-expansive and $f : (M, \Boxed\sqsubset) \to (M', \Boxed{\sqsubset'})$ is a rigid surjection.
We let $\Spec(M, d, \Boxed<) = \Spec(M, d)$.

Our starting point is a result from \cite{masul-drthms-relstru}
that the category of finite linearly ordered metric spaces and non-expansive rigid surjections has the dual Ramsey property.

\begin{LEM}\label{dthms.lem.met}\cite{masul-drthms-relstru}
  Let $\LMetNersFin$ denote the category of finite linearly ordered metric spaces and non-expansive rigid surjections.
  Let $\CC$ be a full subcategory of $\LMetNersFin$ such that for every positive integer $m$ and every $\delta \in \Spec(\CC)$
  there is an integer $n \ge m$ such that $\calE^<_{n,\delta} \in \Ob(\CC)$. Then
  $\CC$ has the dual Ramsey property. In particular, the category $\LMetNersFin$ has the dual Ramsey property.
\end{LEM}

We can now upgrade this result and actually compute dual small Ramsey degrees of finite metric spaces
with respect to non-expansive surjections.

\begin{THM}
  The class of finite metric spaces has dual small Ramsey degrees with respect to non-expansive surjections.
  Actually, $t^\partial(\calM) = |M|!$ for every finite metric space $\calM = (M, d)$.
\end{THM}
\begin{proof}
  Just as a notational convenience let $\CC = \MetNes^\fin$, let $\CC^* = \LMetNersFin$ and
  let $U : \CC^* \to \CC$ be the forgetful functor that forgets the ordering: $U(M, d, \Boxed<) = (M, d)$ on objects
  and $U(f) = f$ on morphisms.

  Let us show that $U$ is a dually reasonable precompact expansion with unique dual restrictions and the dual expansion property.
  It is clear that $U$ is a precompact expansion because $|U^{-1}(M, d)| = |M|!$ for every finite metric space $(M, d)$.

  To see that $U$ is dually reasonable take any linearly ordered finite metric space $(M, d, \Boxed<)$
  and any non-expansive surjective map $f : (M', d') \to (M, d)$. Then it is easy to find a linear ordering $<'$ of $M'$
  so that $f : (M', \Boxed{<'}) \to (M, \Boxed<)$ is a rigid surjection. This turns~$f$ into a non-expansive rigid surjection
  $f : (M', d', \Boxed{<'}) \to (M, d, \Boxed<)$.

  To see that $U$ has unique dual restrictions take any linearly ordered finite metric space $(M, d, \Boxed<)$
  and any non-expansive surjective map $f : (M, d) \to (M', d')$. Then there is a unique linear ordering $<'$ of $M'$
  so that $f : (M, \Boxed<) \to (M', \Boxed{<'})$ is a rigid surjection, which turns~$f$ into a non-expansive rigid surjection
  $f : (M, d, \Boxed<) \to (M', d', \Boxed{<'})$.

  Finally, let us show that $U$ has the dual expansion property. Take any finite metric space $(M, d)$, let
  $\delta = \max(\Spec(M, d))$ and $n = |M|$. Then $\calE_{n, \delta}$ is the witness for the dual expansion property
  for $(M, d)$. Namely, take any linear ordering $\prec$ of $M$ and any linear ordering $\sqsubset$ of $\{1, 2, \ldots, n\}$.
  Then it is easy to see that any rigid surjection $(\{1, 2, \ldots, n\}, \Boxed\sqsubset) \to (M, \Boxed\prec)$ is also
  a non-expansive map $\calE_{n,\delta} \to (M, d)$.

  In order to apply the dual of Theorem~\ref{sbrd.thm.small2}~$(b)$ we still have to show
  that $\CC^*$ is dually directed. Take two finite linearly ordered metric spaces
  $(M_1, d_1, \Boxed{<_1})$ and $(M_2, d_2, \Boxed{<_2})$, let $n = \max\{|M_1|, |M_2|\}$ and let\break
  $\delta = \max(\Spec(M_1, d_1) \cup \Spec(M_2, d_2))$.
  Then it is easy to construct non-expansive rigid surjections $\calE^<_{n, \delta} \to (M_1, d_1, \Boxed{<_1})$
  and $\calE^<_{n, \delta} \to (M_2, d_2, \Boxed{<_2})$.
  Therefore, the dual of Theorem~\ref{sbrd.thm.small2}~$(b)$ applies and
  $$
    t^\partial_\CC(M, d) = \sum_{\substack{\Boxed< \text{ is a linear} \\ \text{ordering of } M}} t^\partial_{\CC^*}(M, d, \Boxed<).
  $$
  Note that $t^\partial_{\CC^*}(M, d, \Boxed<) = 1$ because of Lemma~\ref{dthms.lem.met}, so
  $t^\partial_\CC(M, d) = |M|!$, the number of all linear orderings of~$M$.
\end{proof}

Now that we have settled the issue of small Ramsey degrees, let us move on to big  Ramsey degrees.
We shall compute dual big Ramsey degrees of finite metric spaces in the metric space
$\Omega = (\omega, d_\omega)$, which is the metric space on $\omega$ where $d_\omega(x, y) = \max\{x, y\}$ if $x \ne y$ and $d_\omega(x, x) = 0$.
The following lemmas justify our choice.

\begin{LEM}\label{drdrs.lem.Omega-univ}
  The metric space $\Omega$ is projectively universal for all finite and countably infinite metric spaces.
\end{LEM}
\begin{proof}
  Let $(M, d)$ be a finite or countably infinite metric space. Let us enumerate $M$ as $M = \{x_0, x_1, x_2, \ldots\}$ and let us construct
  $q : \omega \to M$ as follows. Put $n_0 = 0$ and $q(0) = x_0$. Let $\ell_1 = d(x_0, x_1)$ and choose $n_1 \in \omega$ so that $n_1 > \ell_1$.
  Then put $q(n_1) = x_1$ and $q(n_1 - 1) = \ldots = q(1) = 0$. Assume that $n_0$, $n_1$, \ldots, $n_k$ have been chosen so that
  $q(n_j) = x_j$, $0 \le j \le k$, and $q(i) = x_j$ whenever $n_j < i < n_{j+1}$, $0 \le j < k$.
  Let $\ell_{k+1} = \max(\{n_k\} \cup \{d(x_i, x_j) : 0 \le i < j \le k\})$ and choose $n_{k+1} \in \omega$ so that $n_{k+1} > \ell_{k+1}$.
  Then put $q(n_{k+1}) = x_{k+1}$ and $q(i) = k_k$ for $n_k < i < n_{k+1}$.
  If $M$ is finite this construction stops at some point and then we map all the remaining elements of $\omega$ onto the last $x_s$ in the enumeration
  of~$M$. Anyhow, it is easy to see that $q$ is a non-expansive surjection.
\end{proof}

\begin{LEM}\label{drdrs.lem.OmegaU}
  A countable metric space $\calU$ is projectively universal for all finite metric spaces if and only if
  it is bi-projectible with $\Omega$.
\end{LEM}
\begin{proof}
  $(\Leftarrow)$
  Follows directly from Lemma~\ref{drdrs.lem.Omega-univ}.

  $(\Rightarrow)$
  Let $\calU = (U, d_U)$ be a countable projectively universal metric space.
  Put $U_0^{(0)} = U$ and take any $x_0 \in U_0^{(0)}$.
  Take any $\ell_1 \in \RR$ with $\ell_1 > 1$. Then there is a non-expansive map $q_1 : \calU \to \calE_{2, \ell_1}$,
  where $\calE_{2, \ell_1}$ is the metric space with two points at the distance~$\ell_1$. Then $\ker q_1$ partitions $U$
  into two blocks, so let $U_0^{(1)}$ be the block which contains $x_0$ and let $U_1^{(1)}$ be the other block.
  Take any $x_1 \in U_1^{(1)}$ and note that
  $$
    (\forall a \in U_0^{(1)})(\forall b \in U_1^{(1)})\; d_U(a, b) \ge \ell_1.
  $$
  Let us inductively choose $x_2, x_3, \ldots \in U$ as follows. Assume that $x_0$, $x_1$, \ldots, $x_{n-1}$ have been chosen
  together with distances $\ell_0 = 0$, $\ell_1 > 1$, \ldots, $\ell_{n-1} > n-1$ and
  a partition $\{U_0^{(n-1)}, U_1^{(n-1)}, \ldots, U_{n-1}^{(n-1)}\}$ of $U$ so that:
  \begin{itemize}
    \item $x_i \in U_i^{(n-1)}$, $0 \le i < n$, and
    \item whenever $0 \le i < j < n$ we have that
          $$
            (\forall a \in U_i^{(n-1)})(\forall b \in U_j^{(n-1)})\; d_U(a, b) \ge \ell_j.
          $$
  \end{itemize}
  Take any $\ell_n \in \RR$ with
  $$
    \ell_n > \max(\{n\} \cup \{d_U(x_i, x_j) : 0 \le i < j < n\}),
  $$
  and let $q_n : \calU \to \calE_{2, \ell_n}$ be a non-expansive map. Then $\ker q_n$ partitions $U$
  into two blocks, and the choice of $\ell_n$ ensures that the points $x_0, x_1, \ldots, x_{n-1}$ belong to the same block.
  Let $V^{(n)}$ be the block which contains $x_0, x_1, \ldots, x_{n-1}$, let $U_n^{(n)}$ be the other block, and let
  $$
    U_i^{(n)} = U_i^{(n-1)} \cap V^{(n)}, \quad 0 \le i < n.
  $$
  Let $x_n$ be an arbitrary element of $U_n^{(n)}$. It is easy to verify that:
  \begin{itemize}
    \item $x_i \in U_i^{(n)}$, $0 \le i \le n$
    \item $U_i^{(n-1)} \supseteq U_i^{(n)}$, $0 \le i < n$, and
    \item $(\forall a \in U_i^{(n)})(\forall b \in U_j^{(n)})\; d_U(a, b) \ge \ell_j$ whenever $0 \le i < j \le n$.
  \end{itemize}
  
  Finally, for $n \ge 0$ put
  $$
    U_n = \bigcap_{j = n}^\infty U_n^{(j)}.
  $$
  Note that, by construction, the following holds:
  \begin{itemize}
    \item $x_n \in U_n$ for all $n \ge 0$,
    \item for all $n \ge 0$:
    \begin{equation}\label{drdrs.eq.opadajucilanac}
      U_n^{(n)} \supseteq U_n^{(n+1)} \supseteq U_n^{(n+2)} \supseteq \ldots,
    \end{equation}
    \item $U_i \cap U_j = \0$ for $i \ne j$, and
    \item for all $0 \le i < j$ and all $a \in U_i$, $b \in U_j$:
    \begin{equation}\label{drdrs.eq.ellj}
      d_U(a, b) \ge \ell_j > j = \max\{i, j\}.
    \end{equation}
  \end{itemize}

  Let us show that
  $
    \bigcup_{n=0}^\infty U_n = U
  $.
  Take any $y \in U$.
  If there exist infinitely many $k \in \omega$ such that $y \in U_k^{(k)}$, then there are infinitely many $k \in \omega$
  such that $d_U(x_0, y) \ge \ell_k > k$, which is impossible.
  Therefore, there are only finitely many $k \in \omega$ such that $y \in U_k^{(k)}$.
  Choose an $s \in \omega$ such that $y \notin U_t^{(t)}$ for all $t \ge s$. By construction, if $y \notin U_t^{(t)}$
  then it must be the case that $y \in \bigcup_{i=0}^{t-1} U_i^{(t)}$, whence follows that
  \begin{align*}
    y &\in \bigcap_{t=s}^\infty \bigcup_{i=0}^{t-1} U_i^{(t)} =\\
    &=    (U_0^{(s)} \cup U_1^{(s)} \cup \ldots \cup U_{s-1}^{(s)})\\
    &\cap (U_0^{(s+1)} \cup U_1^{(s+1)} \cup \ldots \cup U_{s-1}^{(s+1)} \cup U_s^{(s+1)})\\
    &\cap (U_0^{(s+2)} \cup U_1^{(s+2)} \cup \ldots \cup U_{s-1}^{(s+2)} \cup U_s^{(s+2)} \cup U_{s+1}^{(s+2)})\\
    &\cap \ldots
  \end{align*}
  In each row we have a union of pairwise disjoint sets, and due to \eqref{drdrs.eq.opadajucilanac}
  in each ``column'' we have a non-increasing sequence of sets (with respect to set inclusion), so
  \begin{align*}
    y &\in \bigcap_{t=s}^\infty \bigcup_{i=0}^{t-1} U_i^{(t)} =\\
    &=    (U_0^{(s)} \cap U_0^{(s+1)} \cap U_1^{(s+2)} \cap \ldots) && [\subseteq U_0 \text{ due to \eqref{drdrs.eq.opadajucilanac}}]\\
    &\cup (U_1^{(s)} \cap U_1^{(s+1)} \cap U_1^{(s+2)} \cap \ldots) && [\subseteq U_1 \text{ due to \eqref{drdrs.eq.opadajucilanac}}]\\
    &\cup \ldots\\
    &\cup (U_{s-1}^{(s)} \cap U_{s-1}^{(s+1)} \cap U_{s-1}^{(s+2)} \cap \ldots) && [\subseteq U_{s-1} \text{ due to \eqref{drdrs.eq.opadajucilanac}}]\\
    &\cup (U_{s}^{(s+1)} \cap U_{s}^{(s+2)} \cap U_{s}^{(s+3)} \cap \ldots) && [\subseteq U_s \text{ due to \eqref{drdrs.eq.opadajucilanac}}]\\
    &\cup (U_{s+1}^{(s+2)} \cap U_{s+1}^{(s+3)} \cap U_{s+1}^{(s+4)} \cap \ldots) && [\subseteq U_{s+1} \text{ due to \eqref{drdrs.eq.opadajucilanac}}]\\
    &\cup \ldots
  \end{align*}
  Therefore, $y \in \bigcup_{n=0}^\infty U_n$.

  Consider the mapping $q : U \to \omega$ defined so that $q(x) = j$ if and only if $x \in U_j$. This is a well-defined surjection
  which is a non-expansive map $\calU \to \Omega$ because of \eqref{drdrs.eq.ellj}. By Lemma~\ref{drdrs.lem.Omega-univ} there is a
  non-expansive map $\Omega \to \calU$, so $\Omega$ and $\calU$ are bi-projectible.
\end{proof}

\begin{LEM}\label{drdrs.lem.biprj-met-spc}
  Let $\calU$ be a countable metric space projectively universal for all finite metric spaces. Then
  $T^\partial(\calM, \calU) = T^\partial(\calM, \Omega)$ for every finite metric space $\calM$.
\end{LEM}
\begin{proof}
  Follows directly from Lemmas~\ref{drdrs.lem.bi-proj-general} and~\ref{drdrs.lem.OmegaU}.
\end{proof}

Therefore, computing dual big Ramsey degrees of finite metric spaces in $\Omega$ leads to no loss of generality
since the results will hold for all projectively universal countable metric spaces.

Let $\LMetNers^{\le\omega}$ denote the category whose objects are finite or countably infinite linearly ordered metric spaces
where the ordering of the infinite metric spaces is of order type $\omega$, and whose morphisms are non-expansive rigid surjections.
Let $\Omega^< = (\omega, d_\omega, \Boxed{<})$ be the linearly ordered metric space based on $(\omega, d_\omega)$ where $<$ is the usual ordering of $\omega$.

\begin{LEM}\label{drdrs.lem.BigT-Mord=1}
  In the category $\LMetNers^{\le\omega}$ we have that $T^\partial(\calM^<, \Omega^<) = 1$ for every finite linearly ordered metric
  space~$\calM^< = (M, d_M, \Boxed{<})$.
\end{LEM}
\begin{proof}
  Just as a notational convenience put $\CC = \LMetNers^{\le\omega}$.
  Take any Borel coloring $\chi : \hom_\CC(\Omega^<, \calM^<) \to k$ and define a coloring
  $\gamma : \RSurj(\omega, (M, \Boxed<)) \to k$ by $\gamma(f) = \chi(f)$ for $f \in \hom_\CC(\Omega^<, \calM^<)$ and
  $\gamma(f) = 0$ otherwise. Then $\gamma$ is also a Borel coloring because $\Phi : \hom_\CC(\Omega^<, \calM^<) \to \RSurj(\omega, (M, \Boxed<)) :
  f \mapsto f$ is injective and continuous and hence takes Borel sets to Borel sets.
  By the Carlson-Simpson Theorem there is a $w \in \RSurj(\omega, \omega)$ such that
  $|\gamma(\RSurj(\omega, (M, \Boxed<)) \circ w)| = 1$. But then $|\chi(\hom_\CC(\Omega^<, \calM^<) \circ w)| = 1$ because
  $\hom_\CC(\Omega^<, \calM^<) \subseteq \RSurj(\omega, (M, \Boxed<))$ and $\chi(f) = \gamma(f)$ for $f \in \hom_\CC(\Omega^<, \calM^<)$.
  To complete the argument note that every rigid surjection $\omega \to \omega$ is a non-expansive map
  $\Omega^< \to \Omega^<$ (actually, $\hom_\CC(\Omega^<, \Omega^<) = \RSurj(\omega, \omega)$).
\end{proof}

\begin{LEM}\label{drdrs.lem.Omega-self-sim}
  Let $\MetNes^{\le\omega}$ denote the category of finite or countably infinite metric spaces whose morphisms are
  non-expansive surjections, and let $U : \LMetNers^{\le\omega} \to \MetNes^{\le\omega}$ be the forgetful functor that forgets the ordering.
  Then $\Omega^<$ is dually self-similar with respect to $U$.
\end{LEM}
\begin{proof}
  Let $q : \Omega \to \Omega$ be a non-expansive surjection and let $\restr{\Omega^<}{q} = (\omega, d_\omega, \Boxed\sqsubset)$.
  Note that $\sqsubset$ is a linear ordering of $\omega$ of order type $\omega$ but possibly different from~$<$.
  We have to construct a non-expansive rigid surjection $r : (\omega, d_\omega, \Boxed\sqsubset) \to \Omega^<$.
  \begin{center}
    \begin{tikzcd}
      \Omega^< \arrow[r, leftarrow, "r"] & \restr{\Omega^<}{q} \arrow[r, leftarrow, "q"] \arrow[d, mapsto, dashed, "U"'] & \Omega^< \arrow[d, mapsto, dashed, "U"]\\
         & \Omega \arrow[r, leftarrow, "q"] & \Omega
    \end{tikzcd}
  \end{center}
  Assume that $\sqsubset$ orders $\omega$ as: $x_0, x_1, x_2, \ldots$. There is an $n_0 \ge 0$ such that $x_{n_0} = 0$. Put
  $r(x_i) = 0$ for $0 \le i \le n_0$. Then choose $n_1$ so that $x_{n_1} = \min(\omega \setminus \{x_0, \ldots, x_{n_0}\})$
  and put $r(x_i) = 1$ for $n_0 + 1 \le i \le n_1$. Note that $n_1 \ge 1$ so that the partial function we have defined up to now
  (that will become the full $r$ eventually) is a non-expansive map and a rigid surjection.
  Then choose $n_2$ so that $x_{n_2} = \min(\omega \setminus \{x_0, \ldots, x_{n_1}\})$
  and put $r(x_i) = 2$ for $n_1 + 1 \le i \le n_2$. Note that $n_2 \ge 2$ and that the partial function we have defined up to now
  is still a non-expansive map and a rigid surjection. And so on.
\end{proof}

\begin{THM}
  Consider metric spaces and non-expansive surjective maps,
  and let $\calU$ be a countable metric space projectively universal for all finite metric spaces.
  Then $T^\partial(\calM, \calU) = |M|!$ for every finite metric space $\calM = (M, d_M)$.
\end{THM}
\begin{proof}
  Just as a notational convenience let $\CC = \MetNes^{\le\omega}$.
  Let $\calM = (M, d_M)$ be a finite metric space.  With Lemma~\ref{drdrs.lem.biprj-met-spc} in mind it suffices to show that
  $T^\partial_\CC(\calM, \Omega) = |M|!$.

  Let $\CC^* = \LMetNers^{\le\omega}$ and let $U : \CC^* \to \CC$ be the forgetful functor that forgets the ordering.
  Recall that $\Omega^<$ is dually self-similar with respect to $U$ by Lemma~\ref{drdrs.lem.Omega-self-sim}. It is also easy to see that
  $U$ has unique dual restrictions. So, the dual of Theorem~\ref{sbrd.thm.big1-dual-borel-weaker} applies and we have that
  $$
    T^\partial_\CC(\calM, \Omega) = \sum_{\calM^< \in U^{-1}(\calM)} T^\partial_{\CC^*}(\calM^<, \Omega^<).
  $$
  But $T^\partial_{\CC^*}(\calM^<, \Omega^<) = 1$ by Lemma~\ref{drdrs.lem.BigT-Mord=1}, so
  $T^\partial_\CC(\calM, \Omega) = |U^{-1}(\calM)| = |M|!$.
\end{proof}

\section*{Acknowledgements}

The second author was supported by the Science Fund of the Republic of Serbia, Grant No.~7750027:
Set-theoretic, model-theoretic and Ramsey-theoretic phenomena in mathematical structures: similarity and diversity -- SMART.



\begin{thebibliography}{99}
\bibitem{carlson-simpson-1984}
  T.\ J.\ Carlson, S.\ G.\ Simpson.
  A dual form of Ramsey's theorem.
  Adv.\ Math.\ 53 (1984), 265--290.

\bibitem{dasilvabarbosa}
  K. Dasilva Barbosa.
  A Categorical Notion of Precompact Expansions.
  (Preprint available as arXiv:2002.11751)









\bibitem{graham-leeb-rothschild}
  R. L. Graham, K. Leeb, B. L. Rothschild.
  Ramsey's theorem for a class of categories.
  Advances in Math. 8 (1972), 417--433.


\bibitem{Hubicka-Nesetril-AllThoseRamseyClasses}
  J. Hubi\v cka, J. Ne\v set\v ril.
  All those Ramsey classes (Ramsey classes with closures and forbidden homomorphisms).
  Advances in Mathematics 356 (2019), 106791

\bibitem{jezek-nesetril-1983}
  J. Je\v zek, J. Ne\v set\v ril.
  Ramsey varieties.
  European Journal of Combinatorics 4(1983), 143--147


\bibitem{KPT}
  A.\ S.\ Kechris, V.\ G.\ Pestov, S.\ Todor\v cevi\'c.
  Fra\"\i ss\'e limits, Ramsey theory and topological dynamics of automorphism groups.
  GAFA Geometric and Functional Analysis, 15 (2005) 106--189.

\bibitem{kechris-sokic-todorcevic}
  A.\ Kechris, M.\ Soki\'c, S.\ Todor\v cevi\'c.
  Ramsey properties of finite measure algebras and topological dynamics of the group of measure preserving automorphisms: some results and an open problem,
  pp.\ 69--86. In: Caicedo,  Cummings, Koellner, Larson (editors), Foundations of mathematics, Logic at Harvard,
  Essays in honor of Hugh Woodin’s 60th birthday, (March 27–29, 2015, Harvard University, Cambridge, MA), Contemporary Mathematics 690,
  American Mathematical Society, 2017


\bibitem{masul-drp}
  D.\ Ma\v sulovi\'c.
  A Dual Ramsey Theorem for Permutations.
  Electronic Journal of Combinatorics 24(3) (2017), \#P3.39

\bibitem{masul-drthms-relstru}
    D.\ Ma\v sulovi\'c.
    On dual Ramsey theorems for relational structures.
    Czechoslovak Mathematical Journal 70 (2020), 553--585.

\bibitem{masul-kpt}
    D. Ma\v sulovi\'c.
    The Kechris-Pestov-Todor\v cevi\'c correspondence from the point of view of category theory.
    Applied Categorical Structures 29 (2021), 141--169

\bibitem{masul-rdbas}
  D. Ma\v sulovi\'c.
  Ramsey degrees: big v. small.
  European Journal of Combinatorics, 95 (2021), Article 103323.

\bibitem{masul-drp-algs}
  D.\ Ma\v sulovi\'c.
  Dual Ramsey properties for classes of algebras.
  European Journal of Combinatorics, 112 (2023), Article 103716

\bibitem{masul-rpppg}
  D. Ma\v sulovi\'c.
  Ramsey properties of products and pullbacks of categories and the Grothendieck construction.
  Applied Categorical Structures 31 (2023), article number: 6.

\bibitem{masul-coalg-unary-algs}
  D. Ma\v sulovi\'c.
  Coalgebraic methods for Ramsey degrees of unary algebras.
  Math. Slovaca 74 (2024), 819--834
    
\bibitem{N1995}
  J.\ Ne\v set\v ril.
  Ramsey theory. In: R.\ L.\ Graham, M.\ Gr\"otschel and L.\ Lov\'asz, eds, Handbook of Combinatorics, Vol.~2,
  1331--1403, MIT Press, Cambridge, MA, USA, 1995.




\bibitem{Nesetril-Rodl-MoRT}
  J. Ne\v set\v ril, V. R\"odl (Eds).
  Mathematics of Ramsey Theory.
  Springer-Verlag Berlin Heidelberg 1990




\bibitem{vanthe-more}
  L.\ Nguyen Van Th\'e.
  More on the Kechris-Pestov-Todorcevic correspondence: precompact expansions.
  Fund.\ Math.\ 222 (2013), 19--47


\bibitem{Ramsey}
  F.\ P.\ Ramsey.
  On a problem of formal logic.
  Proc.\ London Math.\ Soc.\ 30 (1930), 264--286.


\bibitem{sokic-semilattices}
  M. Soki\'c.
  Semilattices and the Ramsey property.
  Journal of Symbolic Logic 80 (2015), 1236--1259.

\bibitem{sokic-unary-functions}
  M. Soki\'c.
  Unary functions.
  European J. Combin. 52 (2016), 79--94.

\bibitem{solecki-structs}
  S. Solecki.
  A Ramsey theorem for structures with both relations and functions.
  J. Combin. Theory Ser. A 117 (2010) 704--714.


\end{thebibliography}
\end{document}